\documentclass[12pt]{amsart}

\usepackage[english]{babel}

\usepackage[letterpaper,top=2cm,bottom=2cm,left=3cm,right=3cm,marginparwidth=1.75cm]{geometry}

\usepackage{tikz}
\usetikzlibrary{cd}
\usepackage{amsmath,amssymb}
\usepackage{thmtools}
\usepackage{graphicx}
\usepackage[colorlinks=true, allcolors=blue]{hyperref}
\usepackage[capitalise,nameinlink]{cleveref}

\newcommand{\bC}{{\mathbb{C}}}
\newcommand{\cE}{{\mathbb{E}}}

\newcommand{\bP}{{\mathbb{P}}}
\newcommand{\bZ}{{\mathbb{Z}}}
\newcommand{\bR}{{\mathbb{R}}}
\newcommand{\boldb}{{\mathbf{b}}}

\newcommand{\cX}{{\mathcal{X}}}
\newcommand{\cY}{{\mathcal{Y}}}
\newcommand{\cO}{{\mathcal{O}}}
\newcommand{\cT}{{\mathcal{T}}}

\newcommand{\fL}{{\mathsf{L}}}
\newcommand{\fT}{{\mathsf{T}}}

\newcommand{\tD}{{\widetilde{D}}}

\newcommand{\Si}{{\Sigma}}

\newcommand{\Bl}{{\mathrm{Bl}}}
\newcommand{\Tot}{{\mathrm{Tot}}}
\newcommand{\bSi}{{\mathbf{\Sigma}}}
\newcommand{\conv}{{\mathrm{conv}}}
\newcommand{\disk}{{\mathrm{disk}}}
\newcommand{\Nbd}{{\mathrm{Nbd}}}
\newcommand{\ori}{{\mathfrak{or}}}

\def\C{\mathbb{C}}
\def\Z{\mathbb{Z}}
\renewcommand{\P}{\mathbb{P}}

\def\In{\subset}
\renewcommand{\lim}{\text{lim}}
\def\RM{\backslash}
\def\d{\partial}
\def\wt{\widetilde}

\def\congto{\xrightarrow[]{\sim}}

\DeclareMathOperator{\Fuk}{Fuk}
\DeclareMathOperator{\Coh}{Coh}

\DeclareMathOperator{\CC}{CC}

\newcommand{\Spec}{\operatorname{Spec}}

\newcommand {\Perf}{\mathrm{Perf}}
\newcommand {\Sh}{\mathrm{Sh}}
\newcommand {\Fun}{\mathrm{Fun}}

\renewcommand{\ss}{\operatorname{ss}}
\newcommand{\prl}{\mathsf{Pr}^\L}

\newcommand{\prlstw}{\prl_{\mathrm{st},\omega}}
\newcommand{\prrstw}{\mathsf{Pr}^{\R}_{\mathrm{st},\omega}}
\newcommand{\ex}{{\operatorname{ex}}}
\newcommand{\Hom}{\operatorname{Hom}}

\newcommand{\catex}{\mathsf{Cat}^{\mathrm{ex}}}
\newcommand{\catperf}{\mathsf{Cat}^{\mathrm{perf}}}
\renewcommand{\dim}{\operatorname{dim}}

\newcommand{\mush}{\mu\mathrm{Sh}}

\newcommand{\R}{\mathrm{R}}
\renewcommand{\L}{\mathrm{L}}

\newcommand{\rank}{\mathrm{rank}}

\newcommand{\cC}{\mathcal{C}}
\newcommand{\cD}{\mathcal{D}}
\newcommand{\cF}{\mathcal{F}}

\newcommand{\id}{\mathbf{id}}

\renewcommand{\lim}{\qopname\relax m{\mathbf{lim}}}

\newcommand{\cofib}{\mathbf{cofib}}
\renewcommand{\ker}{\mathbf{ker}}

\newcommand{\GG}{\mathbb{G}}
\newcommand{\PP}{\mathbb{P}}
\newcommand{\RR}{\mathbb{R}}

\newcommand{\ZZ}{\mathbb{Z}}
\renewcommand{\AA}{\mathbb{A}}

\newcommand{\sL}{\mathsf{L}}

\declaretheorem[
  name=Theorem,
  numberwithin=section,
  refname={Thm.,Thms.},
  Refname={Theorem}{Theorems}
]{theorem}

\declaretheorem[
  name=Corollary,
  sibling=theorem,
  refname={Cor.,Cors.},
  Refname={Corollary,Corollaries}
]{cor}

\declaretheorem[
  name=Proposition,
  sibling=theorem,
  refname={Prop.,Props.},
  Refname={Proposition,Propositions}
]{prop}

\declaretheorem[
  name=Lemma,
  sibling=theorem,
  refname={Lem.,Lems.},
  Refname={Lemma,Lemmas}
]{lemm}

\declaretheorem[
  name=Definition,
  sibling=theorem,
  refname={Def.,Defs.},
  Refname={Definition,Definitions}
]{defi}

\declaretheorem[
  name=Example,
  sibling=theorem,
  refname={Ex.,Exs.},
  Refname={Example,Examples}
]{exam}

\declaretheorem[
  name=Remark,
  sibling=theorem,
  refname={Rem.,Rems.},
  Refname={Remark,Remarks}
]{rem}

\title{Homological Mirror Symmetry for Conic Bundle}
\author{Bohan Fang}
\address{Bohan Fang, Beijing International Center for Mathematical Research, Peking University, 5 Yiheyuan Road, Beijing 100871, China}
\email{bohanfang@gmail.com}

\author{Yuze Sun}
\address{Yuze Sun, Beijing International Center for Mathematical Research, Peking University, 5 Yiheyuan Road, Beijing 100871, China}
\email{sunyuze@stu.pku.edu.cn}

\author{Peng Zhou}
\address{Peng Zhou, UC Berkeley, 753 Evans Hall, Berkeley, CA, 94720}
\email{pzhou.math@gmail.com}

\begin{document}

\begin{abstract}
We study the homological mirror symmetry statement where A-side is the conic bundle Hori-Vafa mirror $\mathcal Y=\{uv = f(z)\} \subset \mathbb C^2 \times (\mathbb C^*)^n$ for a Laurent polynomial $f$ in $(\mathbb C^*)^n$, and B-side is some a toric Calabi--Yau $n+2$-fold with a smooth anti-canonical divisor removed ${\mathcal X}^\circ = \mathcal X \setminus w^{-1}(-1)$. We show that when $\mathcal X$ is the canonical bundle of a toric Fano $n$-orbifold $S$ and $f$ is its Givental superpotential, the strong deformation retraction skeleton $ \mathsf{L}$ of $\mathcal Y$ in the sense of \cite{RSTZ14} has a Weinstein neighborhood $U$, such that the wrapped microlocal sheaf category $\mu\mathrm{Sh}^w_{\mathsf L}({\mathsf L})\cong \mathrm{Coh}({\mathcal X}^\circ)$. This proves a microlocal categorical version of the SYZ mirror in \cite[Thm.~1.7] {abouzaid2016lagrangian}. We also extend the definition of characteristic cycles for constructible sheaves in cotangent bundles from \cite[Ch.~IX]{KS90} to finite-rank objects in $\mu\mathrm{Sh}^w_{\mathsf L}({\mathsf L})$, and describe the characteristic cycles for objects mirror to a coherent sheaf supported on $S$.
\end{abstract}

\maketitle

\section{Introduction}
Consider the following toric homological mirror symmetry for $\P^1$, 
$$ \Coh(\P^1) \cong \mathrm{FS}( \C^*, x + 1/x). $$
We can ask, what is the mirror for its canonical bundle $K_{\P^1}$? One answer is given by Givental mirror for Fano toric variety
\begin{equation}\label{FLTZ-mirror}
    \Coh(K_{\P^1})  \cong \mathrm{FS}( \C_x^* \times \C^*_v,\;v^{-1}(x + 1/x-R)), \quad R \gg 0
\end{equation}
where the Fukaya--Seidel category $\mathrm{FS}$ is interpreted as wrapped Fukaya category of Weinstein sector with the stop given by a generic fiber of the superpotential. Another answer is
\begin{equation} \label{alt-mirror}
    \Coh(K_{\P^1})  \cong \mathrm{FS}(\cY, u), \quad \cY =  \{uv=x+1/x-R\} \In \C^2_{uv} \times \C_x,  \quad R \gg 0
\end{equation}
The two mirrors to the same space $K_{\P^1}$ comes from different SYZ fibration in the complement of different anti-canonical divisors, $w^{-1}(0)$ vs $w^{-1}(1)$, where $w: K_{\P^1} \to \C$ is the canonical function that vanishes on the toric boundary.

However, we do not work on \cref{alt-mirror} here, but we can prove a version where B-side removes a divisor and A-side removes the superpotential, thus a wrapped Fukaya category with no stop on a Weinstein manifold.
\begin{prop} \label{pp:intro-ex}
\begin{equation} \label{eq:intro-ex}
     \Coh(K_{\P^1} \RM \{w=-1\} ) \cong \Fuk(\{uv = x+1/x - R\}),  \quad R \gg 0
\end{equation}
    
\end{prop} 
This is proposed in \cite[Thm.~1.7]{abouzaid2016lagrangian}. 
\begin{proof}

We prove this result by colimit gluing on both side. On the A-side, we present the space as a fiber product over $\C$
$$ \cY=\{uv = x+1/x - R\} = \{uv + R/2 = x+1/x - R/2=z\} = \C^*_x \times_{\C_z} \C^2_{u,v} $$
Then 
$$ \cY_L:=\cY|_{z \leq 1}, \quad \cY_R=\cY|_{z \geq -1} $$ 
are Weinstein subsector of $\cY$ for appropriate choices of Weinstein structures on $\cY$. Sectorial descent in \cite{GPS2} gives pushout diagram 
$$ \begin{tikzcd}
  & \Fuk(\cY) & \\
  \Fuk(\cY_L) \ar[ur] & & \Fuk(\cY_R) \ar[ul]\\
  & \Fuk(\cY_L \cap \cY_R)\ar[ur] \ar[ul] & 
\end{tikzcd}
$$
where we can trivialize the factors in the fiber product in $\cY_L, \cY_R$ and get
\begin{equation} \label{A-side-decompose}
    \begin{cases}
     \Fuk(\cY_L) \cong \Fuk(\{x+1/x \leq R/2\} \times \{uv+R/2=0\}) \cong \Coh(\P^1) \otimes \Coh(\C^*) \\
     \Fuk(\cY_R) \cong \Fuk(\{x+1/x = R/2\} \times \{uv + R/2 \geq 0\}) \cong \Coh(\d \P^1) \otimes \Coh(\{1\}) \\
     \Fuk(\cY_L \cap \cY_R) \cong \Coh(\d \P^1) \otimes \Coh(\C^*)
\end{cases}
\end{equation}

Now, using Gammage--Le \cite[Thm.~4.13]{gammage2022mirror}, we can recognize the B-side pushout diagram as
$$ \begin{tikzcd}
  & \Coh(\Bl_{\d \P^1 \times 1}(\P^1 \times \C^*) \RM \wt{\P^1 \times \C^*}) & \\
  \Coh(\P^1 \times \C^*) \ar[ur] & & \Coh(\d \P^1 \times 1) \ar[ul]\\
  & \Coh(\d \P^1 \times \C^*)\ar[ur] \ar[ul] & 
\end{tikzcd}
$$
It is easy to check that
$$ \Bl_{\d \P^1 \times 1}(\P^1 \times \C^*) \RM \wt{\P^1 \times \C^*}\cong K_{\P^1} \setminus w^{-1}(-1). $$ 
\end{proof}

For more general canonical bundle over a compact toric Fano manifold $S$, one could also expects an equivalence between $\Coh(\cX^\circ)$ and $\Fuk(\cY)$. When $S$ is toric del Pezzo surfaces, Seidel proved a full and faithful function from $D^b\Coh_S(\cX^\circ) \hookrightarrow D\Fuk_c(\cY),$ where the left side is the (triangulated) derived category of coherent sheaves supported on $S$, while the right side is the derived Fukaya category of compact exact Lagrangians \cite{seidel10}. Gross--Matessi explored this equivalence for line bundles on $\cX^\circ$ in \cite{grossmatessi2018}.

In the paper, we try to generalize $\bP^1$ to any compact Fano orbifold $S$. There are two problems.

\begin{enumerate}

  \item It is symplectically tricky to produce two Weinstein sectors $\cY_L$ and $\cY_R$. One needs to modify their standard Liouville structures to make them sectors, and then show these modified structures could work with known toric mirror symmetry. Both of these steps are highly non-trivial.
  
  \item On the algebraic geometry side, we need to generalize Gammage--Le's argument for singular center. 
\end{enumerate}

For 1) we simply work with microlocal sheaves on a skeleton in $\cY$, and bypass this problem, while avoid claiming we deal with Fukaya categories. For 2) we will prove the desired result for a singular center on the B-side in \cref{sec:b-model}.

\begin{theorem}
  \label{thm:main}
Let $N$ be a lattice of rank $n$. Let $S$ be a compact Fano toric orbifold with toric stacky fan $\bSi$ in $N_\RR$, and with ray generators $\mathbf{b}=\{v_1,\cdots,v_m\}\In N$. Let $\cX=K_S$ be the canonical bundle of $S$ with stacky fan $\wt\Sigma\In N_\RR\oplus\RR$, and ray generators $\{(v,1)\mid v\in\mathbf{b}\}\In N\times\Z$. Let $w:\cX\to\C$ be the function corresponding to the character $(\vec 0,1)$. Let
\[
  f=\sum_{v\in\Sigma^{(1)}}a_v z^v : (\C^*)^n\cong\Hom(N,\C^*)\to\C,
\]
and let $\cY$ be the conic bundle $\{xy=f(z)-R\}$ for a generic $R>0$. When $a_v>0$ are generic, the RSTZ skeleton $\fL$ of $\cY$ is a strong deformation retract of $\cY$, and there exists a Weinstein thickening $U$ with core $\fL$ such that there is an equivalence
\[
\kappa_{\cX^\circ}: \Coh(\cX^\circ)\simeq\mush^w_\fL(\fL),
\]
where $\mush^w$ is the cosheaf of wrapped microlocal categories defined in \cite{Nad16,shende2021,nadlershende20}, and $\cX^\circ=\cX \setminus \{w=-1\}$. 
\end{theorem}

The RSTZ skeleton \cite{RSTZ14} is the combinatorial description of a strong deformation retraction for an  hypersurface in $\bC^a\times (\bC^*)^b$. When there are only torus components, i.e. $a=0$, with a non-standard choice of a Weinstein structure on $\cY$, RSTZ skeleton could be made into the core of $\cY$ \cite{shende-gammage,Peng-hypersurface}.

This theorem could be thought as a version of coherent-constructible correspondence (CCC) for $\cX^\circ$. The usual coherent constructible correspondence \cite{kuwagaki20} for a toric orbifold $\cX$ is an equivalence $\kappa_\cX:\Coh(\cX)\stackrel{\sim}{\longrightarrow}\mush^w_\Lambda(\Lambda)$ where $\Lambda$ is a conical, piecewise linear Lagrangian in a cotangent bundle (FLTZ skeleton \cite{FLTZ-variety}).

\begin{exam}
For $S=\PP^1$, the FLTZ skeleton is obtained from the zero section circle
$Q\simeq S^1$ by attaching two rays at the same point in $S^1$ one ray for each toric boundary point.  Hence
$\fL$ can be drawn as a torus $Q\times S^1$ with two meridian disks attached:
\[
  \fL=
  (\Lambda_{\PP^1}\times S^1)
  \cup_{\partial\Lambda_{\PP^1}\times S^1\times I}
  (\partial\Lambda_{\PP^1}\times D_\disk),
  \qquad
  \partial\Lambda_{\PP^1}=\{p_0,p_\infty\}.
\]
See Figure \ref{fig:P^1}.
\begin{figure}[htbp]
\centering
\includegraphics[width=0.58\textwidth]{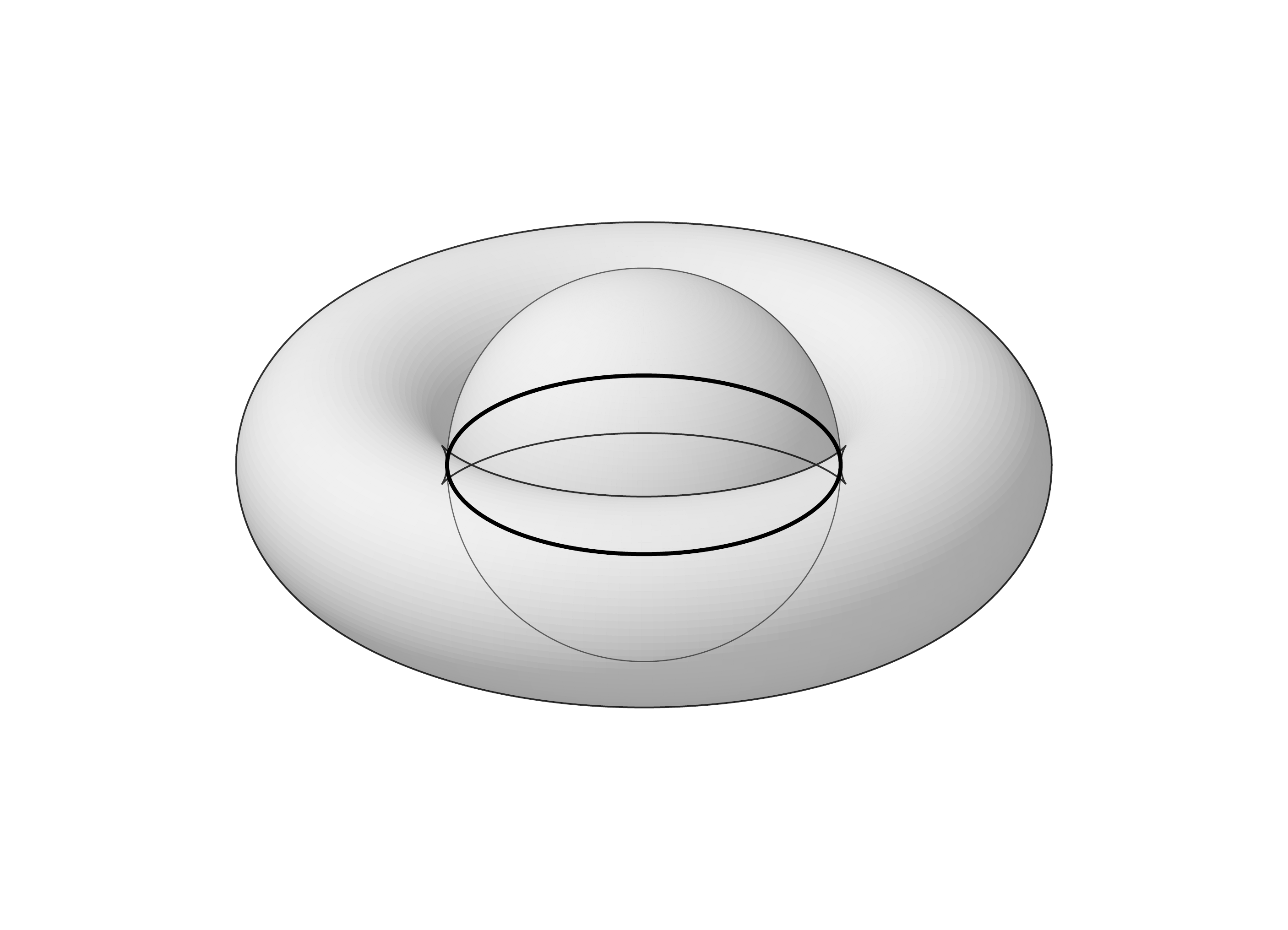}
\caption{The skeleton $\fL$ for $K_{\bP^1}$ is a torus glued with a sphere along an longitude/big circle. The torus is $Q\times S^1$ where $Q=M_\bR/M\cong S^1$. The upper and lower hemisphere of the sphere is $p_\infty \times D_{\mathrm{disk}}$ and $p_0\times D_{\mathrm{disk}}$ respectively.}
\label{fig:P^1}
\end{figure}

\end{exam}

\begin{rem}
We do not claim, but in good situation do expect $\mush^w_\fL(\fL)$ is equivalent to the wrapped Fukaya category $\Fuk(\cY)$. By \cite{GPS3}, one has to show $\fL$ is the Liouville core of $\cY$ with certain Weinstein structure. This problem is similar to identifying the piecewise combinatorial data in the mirror of $\bSi$ in terms of skeleton \cite{FLTZ-variety,FLTZDM} with the superpotential $f$. One expects to use a different Liouville structure from the standard one as an affine hypersurface, and needs to impose more condition than Fano.
See \cite{Peng-hypersurface,shende-gammage} for more details. 
\end{rem}

\begin{rem}
A general semi-projective toric Calabi-Yau orbifold $\cX$ is always derived equivalent to $K_S$ for some $S$. One simply replaces the triangulation of its defining polytope by a star triangulation.
\end{rem}

We also introduce a notion of characteristic cycles \cref{thm:cc-global-ambient} in \cref{sec:characteristic-cycles}, which is a generalization of the characteristic cycles in the cotangent bundle as defined in \cite[Ch.~IX]{KS90}. Namely, for a finite microlocal rank object $\cF$ in $\mush_\fL(\fL)$, we get a piecewise analytic Lagrangian cycle $\CC(\cF)$ in $U$. It could be regarded as topological approximation of the Lagrangian brane in $\Fuk(U)$ corresponding to $\cF$ under the microlocal-Fukaya equivalence \cite{GPS3}. We show the following theorem.

\begin{theorem}[\cref{thm:cap-formula}] 
Let $\kappa_{\cX^\circ}$ and $\kappa_S$ be the microlocal mirror CCC functors for $\cX^\circ$ and $S$ respectively. We have
\[
  \CC\bigl(\kappa_{\cX^\circ}(\iota_*G)\bigr)=
  \operatorname{Cap}\Bigl(\CC\bigl(\kappa_S(G)\bigr)\Bigr).
\]
\end{theorem}
The notion $\CC$ on right side is the usual Kashiwara--Schapira's characteristic cycle, 
and the $\mathrm{Cap}$ operation glues its infinity boundary with a disk
(\cref{sec:cap}).

\subsection{Structure of the paper}

We review the categorical notion in \cref{sec:notion}. In \cref{sec:setup}, we introduce the definition of $\cX=K_S$ for a compact toric Fano orbifold $S$, and define $\cX^\circ$ and the Hori-Vafa mirror $\cY$. We furthre compute the RSTZ skeleton of $\cY$ and construct the A-model microlocal sheaf category for the skeleton.

In Section \cref{sec:b-model}, we study the category of coherent sheaves on $\cX^\circ$ and prove a pushout diagram \cref{B-side-colim}. In Section \cref{sec:cap}, by identifying the pushout diagrams from A and B-sides, we prove the equivalence $\kappa_{\cX^\circ}$.

In  \cref{sec:characteristic-cycles}, we provide a general construction of characteristic cycles for $\mush^c$, and prove the cap formula \cref{thm:cap-formula} for characteristic cycles of a coherent sheaf supported on the base.

\subsection*{AI statements}

The work of this paper is substantially assisted by AI. We mainly use the Web version of \textsf{GPT-5 Pro}. In particular, the construction of the characteristic cycles was proposed by \textsf{GPT}. With heavy human correction and prompts, we eventually obtained the current definition (\cref{thm:cc-global-ambient}). The formula of the characteristic cycle \cref{thm:cap-formula}  was also proposed by \textsf{GPT}, but the AI did not produce any repairable proof.

All AI generated texts are reviewed and revised by the authors. Some sample interaction with \textsf{GPT} is at \url{https://github.com/bohanfang/hms-toric-cy-sheaf-ai-inquiry-sample}.

\subsection*{Acknowledgments} Bohan Fang would like to thank Chiu-Chu Melissa Liu, Song Yu and Zhengyu Zong for valuable discussion. Yuze Sun would like to thank Tatsuki Kuwagaki and Yixiao Li for helpful discussion. The work of BF and YS are partially supported by by National Key R\&D Program of China 2023YFA1009803, NSFC 12125101, NSFC 11890661 and NSFC 11831017.

\section{Categorical notion} 
\label{sec:notion}

In this section we recall the microlocal sheaf formalism in cotangent bundles and on Weinstein manifolds.
Our conventions follow Kashiwara--Schapira \cite{KS90}, Nadler \cite{Nad16}, Nadler--Shende \cite{nadlershende20}, and Ganatra--Pardon--Shende \cite{GPS3}.
Throughout we fix an algebraically closed field $k$ of characteristic $0$.
\smallskip

\begin{defi}\label{def:categorical-conventions}
\begin{enumerate}
    \item Let $\prlstw$ be the category of compactly generated stable categories and colimit-preserving exact functors which preserve compact objects.
    \item Let $\prrstw$ be the same objects, with morphisms those functors which admit both left and right adjoints.
    \item Let $\widehat{\catex}$ be the category of (not necessarily small) stable categories and exact functors.
    \item Let $\catperf$ be the category of essentially small idempotent complete stable categories and exact functors.
\end{enumerate}
\end{defi}

\smallskip

\begin{defi}[Sheaves and singular support {\cite[Ch.~V]{KS90}}]\label{def:sheaves-and-ss}
Let $M$ be a real analytic manifold.
\begin{enumerate}
\item We write $\Sh(M)$ for the $k$-linear category of sheaves of $k$-modules on $M$.
\item For a closed conic subset $K\subset T^*M$, we write $\Sh_K(M)\subset \Sh(M)$ for the full subcategory of objects with singular support contained in $K$.
\item We write $\Sh^c(M)\subset \Sh(M)$ for the full subcategory of constructible objects with perfect stalks, and $\Sh^c_K(M):=\Sh^c(M)\cap \Sh_K(M)$.
\end{enumerate}
\end{defi}

For a compactly generated category $\cC$ we write $\cC^{\omega}$ for its full subcategory of compact objects.
When $\Sh_K(M)$ is compactly generated, we also write $\Sh^w_K(M):=(\Sh_K(M))^\omega$.

\smallskip

\begin{defi}[Microlocalization on cotangent bundles {\cite[\S3]{Nad16}, \cite[\S6]{nadlershende20}}]\label{def:mush-cotangent}
Let $\mathrm{Op}^{\RR_{>0}}_{T^*M}$ be the poset of $\RR_{>0}$-invariant open subsets of $T^*M$.\begin{enumerate}
    \item The large microlocal presheaf is
\[
  \mush^{\mathrm{pre}}:\ \Omega\longmapsto
  \Sh(M)\big/\Sh_{T^*M\setminus \Omega}(M),
\]
viewed as a functor $\mathrm{Op}^{\RR_{>0}}_{T^*M}\to \widehat{\catex}$.
Its sheafification is denoted $\mush$.
    \item The finite microlocal-rank subcategory is the full subsheaf
\[
  \mush^c\subset \mush
\]
of objects whose microstalks are perfect complexes.
    \item The wrapped microlocal category is obtained by taking compact objects:
\[
  \mush^w(\Omega):=\bigl(\mush(\Omega)\bigr)^{\omega}.
\]
\end{enumerate}
\end{defi}

\begin{defi}\label{def:ss-mush}
Let $\Omega\subset T^*M$ be conic open and let $F\in \mush(\Omega)$.
For a point $p\in \Omega$ write $F_p\in \mush_p$ for the stalk of the sheaf of categories $\mush$ at $p$.
Define the microsupport of $F$ inside $\Omega$ by
\[
  \ss(F):=\{p\in \Omega\mid F_p\neq 0\in \mush_p\}.
\]
\end{defi}

\smallskip

\begin{defi}\label{def:subanalytic-isotropic}
Let $(X,\omega)$ be a real analytic symplectic manifold.
A closed subset $\Lambda\subset X$ is called \emph{subanalytic isotropic} if it is subanalytic and admits a locally finite Whitney stratification
$\Lambda=\bigsqcup_\alpha \Lambda_\alpha$
such that each stratum $\Lambda_\alpha$ is a smooth isotropic submanifold, namely $\omega|_{T\Lambda_\alpha}=0$.
\end{defi}

\begin{defi}\label{def:isotropic-in-cosphere}
Let $S^*M$ be the cosphere bundle with its standard contact structure.
A closed subset $\Lambda^\infty\subset S^*M$ is called \emph{subanalytic isotropic} if its conic lift
\[
  \widehat{\Lambda^\infty}:=\bigl(\RR_{>0}\cdot \Lambda^\infty\bigr)\cup M\subset T^*M
\]
is a conic subanalytic isotropic subset in the sense of \cref{def:subanalytic-isotropic}.
\end{defi}
\begin{defi}[Microlocal categories with support condition {\cite[\S3]{Nad16}, \cite[\S6]{nadlershende20}}]\label{def:mush-Lambda}
Let $\Lambda^\infty\subset S^*M$ be a closed subset.
Define a full subsheaf $\mush_{\Lambda^\infty}\subset \mush$ by
\[
  \mush_{\Lambda^\infty}(\Omega)
  :=
  \{F\in \mush(\Omega)\mid \ss(F)\subset \widehat{\Lambda^\infty}\cap \Omega\}
\]
for conic open $\Omega\subset T^*M$.
Set
\[
  \mush^c_{\Lambda^\infty}(\Omega):=\mush^c(\Omega)\cap \mush_{\Lambda^\infty}(\Omega),
  \qquad
  \mush^w_{\Lambda^\infty}(\Omega):=\bigl(\mush_{\Lambda^\infty}(\Omega)\bigr)^{\omega}.
\]
\end{defi}

Equivalently, one can start from the presheaf
\[
  \mush^{\mathrm{pre}}_{\Lambda^\infty}:\ \Omega\longmapsto
  \Sh_{(T^*M\setminus \Omega)\cup \widehat{\Lambda^\infty}}(M)\big/\Sh_{T^*M\setminus \Omega}(M)
\]
and then sheafify.

\begin{prop}[{\cite[\S3]{Nad16}, \cite[\S6]{GPS3}}]\label{thm:basic-mush-cotangent}
Assume $\Lambda^\infty\subset S^*M$ is subanalytic isotropic.
\begin{enumerate}
  \item The sheafification of the presheaf $\mush^{\mathrm{pre}}_{\Lambda^\infty}$ is $\mush_{\Lambda^\infty}$.
  \item For any conic open $\Omega\subset T^*M$, the category $\mush_{\Lambda^\infty}(\Omega)$ is compactly generated, and restriction functors admit both left and right adjoints. That is, $\mush_{\Lambda^\infty}$ takes values in $\prrstw$.
  \item The assignment $\Omega\mapsto \mush^c_{\Lambda^\infty}(\Omega)$ is a sheaf of stable categories.
\end{enumerate}
\end{prop}

\begin{proof}
(1) For an object $F\in \mush^{\mathrm{pre}}(\Omega)$, the set $\ss_\Omega(F):=\ss(F)\cap \Omega$ is well-defined.
Then \cite[Thm.~6.1.2]{KS90} identifies
\[
  \mush_p=\mush^{\mathrm{pre}}_p=\Sh(M)\big/\Sh_{T^*M\setminus \RR_{>0}\cdot p}(M).
\]

(2) $\mush_{\Lambda^\infty}(\Omega)$ is compactly generated by \cite[Lem.~3.15]{Nad16}; the restriction functors preserve products and coproducts by \cite{kuo_Spherical}.

(3) This is local on $\Omega$.
In each cotangent chart it is the traditional microlocal sheaf of objects with perfect microstalks, which is a sheaf by the usual microlocal descent construction.
\end{proof}

Let $\Lambda\subset T^*M$ be a closed conic subset containing the zero section.
By abuse of notation, we also write $\mush_\Lambda$, $\mush^c_\Lambda$, and $\mush^w_\Lambda$ for the categories associated with $\Lambda^\infty:=\Lambda\cap S^*M$. As they are supported on $\Lambda$, we also regard them as sheaves on $\Lambda$.

\begin{prop}\label{prop:mush-restrict-zero}
There are canonical equivalences
\[
  \mush_{\Lambda}(\Lambda)\simeq \Sh_{\Lambda}(M),
  \qquad
  \mush^c_{\Lambda}(\Lambda)\simeq \Sh^c_{\Lambda}(M),
  \qquad
  \mush^w_{\Lambda}(\Lambda)\simeq \Sh^w_{\Lambda}(M).
\]
\end{prop}

\begin{proof}
It suffices to prove $\mush|_M=\Sh$ and $\mush^c|_M=\Sh^c$.
For an open subset $U\subset M$ one has $\mush^{\mathrm{pre}}(T^*U)=\Sh(U)$, so $\mush^{\mathrm{pre}}|_M=\Sh$ is already a sheaf.
The statement for $\mush^c$ is the corresponding finite microlocal-rank subcategory.
Taking compact objects yields the wrapped statement.
\end{proof}
\begin{rem}[sheaf versus cosheaf]\label{rem:sheaf_cosheaf}
Let $\Lambda$ be subanalytic isotropic. By \cite[Notation~5.5.7.7]{HTT}, taking left adjoints exhibits $\mush_\Lambda$ as a cosheaf taking values in $\prlstw$. As taking compact objects induces an equivalence $\prlstw\congto\catperf$, the functor $\mush^w_\Lambda$ is a cosheaf with values in $\catperf$.
\end{rem}
\begin{defi}[Microlocal sheaves on Weinstein manifolds {\cite{shende2021}, \cite{nadlershende20}}]\label{def:mush-weinstein}
Let $(W,\lambda,\phi)$ be a Weinstein manifold, and let $\Lambda\subset W$ be a subanalytic isotropic subset.
Assume $W$ is \emph{stably polarized}, namely we fix a section of $\mathrm{LGr}(TW\oplus \mathbb C^k)$ for some $k$.
Shende constructs an embedding of $W$ into $S^*M$ for some auxiliary manifold $M$, together with a Lagrangian subbundle $\nu_W$ of the symplectic normal bundle, regarded as a submanifold of $S^*M$ via a tubular neighborhood.
One then defines
\[
  \mush^W:=\mush_{\nu_W}|_W,
  \qquad
  \mush^{c,W}:=\mush^c_{\nu_W}|_W,
\]
and for a closed subset $\Lambda\subset W$ one writes $\mush_\Lambda$ and $\mush^c_\Lambda$ for the corresponding subsheaves microsupported in $\Lambda$.
\end{defi}

\begin{prop}[\cite{nadlershende20,GPS3}]
For a stably polarized Weinstein manifold $W$ with a tame skeleton $\Lambda\subset W$, the large microlocal category $\mush_\Lambda$ is a sheaf taking values in $\prrstw$, and $\mush^c_\Lambda$ is a sheaf of stable categories.
If $\Lambda=\mathrm{core}(W)$, then $\mush_\Lambda(\Lambda)$ is invariant under Weinstein homotopy.
\end{prop}

\begin{rem}
When $\Lambda$ is the core of $W$, the wrapped category $\mush^w_\Lambda(\Lambda)$ is the microlocal analogue of the wrapped Fukaya category $\Fuk(W)$.
The above invariance is proved in \cite{nadlershende20}, and also follows from the comparison with wrapped Fukaya categories in \cite{GPS3}.
\end{rem}

\begin{rem}
For simplicity we work with a polarization of $W$, as required in \cite{shende2021,GPS3}.
It induces the grading and orientation data for Fukaya categories \cite{Sei08} and also the Maslov data needed to define $\mush$ and $\mush^c$; see \cite[\S5.3]{GPS3} and \cite{nadlershende20}.
\end{rem}

\section{Mirror A-model to a local toric Calabi-Yau orbifold}

\label{sec:setup}
Let $N\cong \bZ^n$ be a lattice of rank $n$, and let $S$ be a proper Fano toric orbifold with toric stacky fan $\bSi=(\Si,\mathbf b)$ in the sense of \cite{BCS05}. The stacky fan $\bSi$ comes with a fan $\Si$ in $N_\RR$, with $\mathbf b=\{v_1,\cdots, v_m\}\In N$ such that the $1$-cones $\Si(1)=\{\bR_{>0} v_1,\dots,\bR_{>0} v_m\}$. We do not require $\boldb$ to be primitive.

The fan $\bSi$ is equivalently given by an integral convex polytope $P=\mathrm{conv}(v_1,\dots,v_m)$, and a star triangulation $\mathcal T_P$ centered at $0$.
Let 
\[
  f = \sum_{v \in \boldb} a_v z^v: (\C^*)^n\cong \Hom(N, \C^*) \to \C
\]
be the Givental superpotential mirror to $S$.
The conic bundle is often written as $\{xy=f(z)-R\}$. In \cref{sec:skeleton-definition} we use the linearly equivalent presentation obtained from
$
  x=u+iv,\ y=u-iv,
$
namely the hypersurface (Hori-Vafa mirror)
\[
\cY=\{(z,u,v)\in (\bC^*)^n \times \bC^2\mid f(z)-R=u^2+v^2\}.
\]

We impose the positive-phase convention
\[
  R\in\bR_{>0},\qquad a_v\in\bR_{>0}.
\]

We also assume the RSTZ non-degeneracy condition for the polynomial
\[
  -R+\sum_{v\in\boldb}a_vz^v+u^2+v^2
\]
with respect to the toric variety $(\bC^*)^n\times\bC^2$: for every coordinate torus orbit $O$, the scheme-theoretic intersection $\cY\cap O$ is either empty or a smooth reduced hypersurface of $O$, and no coordinate torus orbit is contained in $\cY$.

\subsection{The Lagrangian skeleton $\fL$}
\label{sec:skeleton-definition}
Let 
\[
  Q:=\fT_\bR^\vee\cong (S^1)^n
\]
be the real $n$--torus underlying $(\bC^*)^n$.  We write $T^*Q$ for its cotangent bundle and $S^*Q\subset T^*Q$ for the unit cosphere bundle for an auxiliary metric on $Q$.

Let $\Lambda\subset T^*Q$ be the toric FLTZ conic Lagrangian skeleton mirror to $S$.  It is conic, hence determined by its Legendrian link at infinity
\[
  \partial\Lambda:=\Lambda\cap S^*Q\subset S^*Q.
\]

We recall the explicit (stacky) FLTZ construction.
Let $M:=\Hom(N,\bZ)$ so that $Q=\Hom(N,S^1)\cong M_\RR/M$ and hence $T^*Q\cong Q\times N_\RR$.
For a cone $\sigma\in\Si$ let $N_\sigma\subset N$ be the sublattice generated by $\{v_i\mid \bR_{>0}v_i\subset \sigma\}$.
Following \cite[Def.~6.3]{FLTZDM}, we define
\[
  \Lambda=\Lambda_{\bSi}:=\bigcup_{\sigma\in\Si}\bigl(\sigma^{\perp}_{\bSi}\times (-\sigma)\bigr)\subset T^*Q,
  \qquad
  \sigma^{\perp}_{\bSi}:=\{\theta\in Q\mid \theta(N_\sigma)=1\}.
\]
If the $\bSi$ defines a smooth toric variety, then one recovers the usual FLTZ skeleton.
In the toric setting $\partial\Lambda$ is a compact \emph{piecewise linear} (subanalytic) Legendrian.
We will need a \emph{smooth Weinstein hypersurface} (a ``ribbon'')
\[
  F\subset S^*Q
\]
whose Weinstein skeleton (core) is exactly the given (singular) Legendrian:
\[
  \mathrm{core}(F)=\partial\Lambda.
\]
Thus $F$ is a smooth manifold with boundary equipped with a Liouville form $\lambda_F$  and a Lyapunov function, while its skeleton is allowed to be singular and is identified with the PL Legendrian $\partial\Lambda$.

For a \emph{smooth} Legendrian $L\subset (Y,\alpha)$, the existence of such a ribbon is local: by the Legendrian neighborhood theorem one has a contactomorphism from a neighborhood of $L$ to a neighborhood of the zero section in the $1$-jet space $J^1L=T^*L\times \bR_z$, and the hypersurface $D^*L\times\{0\}\subset J^1L$ is a Weinstein hypersurface with core the zero section; see \cite[Ex.~2.15]{GPS1}.

In our case $\partial\Lambda\subset S^*Q$ is not smooth, but it is the RSTZ Legendrian at infinity associated to the \emph{stacky} fan $\bSi$.  A global ribbon for this \emph{PL} Legendrian is constructed in \cite[Prop.~16]{Shende22}: Shende produces a Weinstein pair $(T^*Q,F)$ whose \emph{relative skeleton} is the conic Lagrangian $\Lambda$ (in the sense recalled in \cite{Shende22}, following \cite{GPS1}).  Since the relative skeleton of a Weinstein pair $(X,F)$ is the union of $\mathrm{core}(X)$ with the Liouville cone on $\mathrm{core}(F)$,
\[
  \mathrm{core}(F)=\partial_\infty\Lambda=\partial\Lambda\subset S^*Q.
\]

Let $D^2$ be a closed disk, let $D_\disk:=\operatorname{int}(D^2)$, and let $I=(0,1)$. We write $\partial D^2=S^1$ and fix an embedding $S^1\times I\hookrightarrow D_\disk$ identifying $S^1\times I$ with an annular collar of the end of the open disk.
Since $\Lambda$ is conical, an end of $\Lambda$ is identified with $\partial\Lambda\times I$.
Define 
\[
\fL= (\Lambda\times S^1) \bigcup_{\partial\Lambda \times S^1\times I} (\partial \Lambda \times D_\disk).
\]

The main result of \cite{RSTZ14} applied here is the following proposition.
\begin{prop}\label{prop:RSTZ-skeleton}
The skeleton $\fL$ is a strong deformation retraction of $\cY$.
\end{prop}
\begin{proof}
We apply the PL skeleton construction of Ruddat--Sibilla--Treumann--Zaslow \cite{RSTZ14}.
The hypersurface is cut out by
\begin{equation}
  \label{eq:rstz-poly}
  -R+\sum_{v\in \boldb}a_v z^v+u^2+v^2=0,
\end{equation}
with $R\in\bR_{>0}$ and $a_v\in\bR_{>0}$.
Set
\[
  \widetilde N:=N\oplus\bZ e_u\oplus\bZ e_v,
  \qquad
  K:=N_\RR\oplus\bR_{\geq 0}e_u\oplus\bR_{\geq 0}e_v\subset \widetilde N_\RR .
\]
Then
\[
  \Spec\bC[K\cap\widetilde N]\cong (\bC^*)^n\times\bC^2,
\]
and the Newton polytope of \eqref{eq:rstz-poly} is
\[
  \Delta=
  \conv\Bigl(\{0\}\cup\{v_1,\ldots,v_m\}\cup\{2e_u,2e_v\}\Bigr)
  \subset K .
\]
This is the Newton polytope used below.

Let
\[
  E:=\conv\{2e_u,2e_v\}\subset \bR e_u\oplus\bR e_v .
\]
Then
$
  \Delta=P*E
$
is the join of an $n$-dimensional polytope and a segment, so
\[
  \dim\Delta=n+2=\dim K.
\]
Moreover the linear part of $K$ is $K^\times=N_\RR$, and the image of $\Delta$ in
\[
  K/K^\times\cong \bR_{\geq0}e_u\oplus\bR_{\geq0}e_v
\]
is $\conv\{0,2e_u,2e_v\}$, whose positive span is the whole cone $\bR_{\geq0}e_u+\bR_{\geq0}e_v$.
Thus the dimension and cone-generation hypotheses in the RSTZ theorem are satisfied for this choice of $\Delta$ and $K$. The non-degeneracy convention in \cref{sec:setup} gives the remaining RSTZ input: $\cY$ is smooth, its intersections with coordinate torus orbits are smooth and reduced, and no coordinate torus orbit is contained in $\cY$.

Let $\cT_P$ be the regular star triangulation of $P$ fixed in \cref{sec:setup}, so that the fan of cones over its nonzero boundary simplices is the given fan $\Sigma$.
Let $\cT_E$ be the triangulation of the segment $E$ with vertices $2e_u$ and $2e_v$.
The join
\[
  \cT_\Delta:=\cT_P*\cT_E
\]
is again a regular star triangulation of $\Delta=P*E$ based at $0$.
Write $\cT$ for the simplices of $\cT_\Delta$ not meeting $0$, with support denoted $\partial\Delta'$ as in \cite[Def.~1.1]{RSTZ14}.
RSTZ associate to the data $(\Delta,\cT,K)$ a topological space
$
  S_{\Delta,\cT,K}
$
and construct a canonical PL embedding
$
  j:S_{\Delta,\cT,K}\hookrightarrow \cY
$
whose image is a strong deformation retract of $\cY$ by the general-cone deformation-retract theorem of \cite[\S5]{RSTZ14}.
It remains to identify $S_{\Delta,\cT,K}$ with $\fL$.

\smallskip

By \cite[Def.~1.1]{RSTZ14}, the space
\[
  S_{\Delta,\cT}\subset \partial\Delta'\times \Hom(\widetilde N,S^1)
\]
consists of pairs $(x,\phi)$ such that $\phi(w)=1$ for every vertex $w$ of the smallest simplex $\tau_x\in\cT$ containing $x$.
Write
\[
  \Hom(\widetilde N,S^1)
  \cong \Hom(N,S^1)\times S^1_u\times S^1_v
  =Q\times S^1_u\times S^1_v,
\]
and write $\phi=(\theta,\alpha,\beta)$, where $\alpha=\phi(e_u)$ and $\beta=\phi(e_v)$.
The vertex conditions are
\begin{equation}
  \label{eq:rstz-constraints-corrected}
  \theta(v_i)=1\quad\text{if }v_i\in\tau_x^{[0]},
  \qquad
  \alpha^2=1\quad\text{if }2e_u\in\tau_x^{[0]},
  \qquad
  \beta^2=1\quad\text{if }2e_v\in\tau_x^{[0]}.
\end{equation}
The first equations are the stacky FLTZ equations: if the toric vertices occurring in $\tau_x$ span a cone $\sigma\in\Sigma$, then
\[
  \{\theta\in Q\mid \theta(v_i)=1\text{ for all such }v_i\}
  =\sigma^\perp_{\bSi}.
\]
The equality uses the definition $\sigma^\perp_{\bSi}=\{\theta\in Q\mid \theta(N_\sigma)=1\}$ and the fact that $N_\sigma$ is generated by the corresponding stacky ray generators.

We now describe the $u,v$-phase factor, including the quotient relation of \cite[Def.~5.6]{RSTZ14} associated to the cone $K$.
Over the relative interior of the segment $E$, both vertices $2e_u$ and $2e_v$ occur, so \eqref{eq:rstz-constraints-corrected} imposes
\[
  \alpha^2=1,
  \qquad
  \beta^2=1.
\]
Thus there are four phase choices $(\alpha,\beta)\in\{\pm1\}\times\{\pm1\}$.
At the endpoint $2e_u$, the smallest face of $K$ containing the $E$-coordinate is $N_\RR\oplus\bR_{\geq0}e_u$.
The RSTZ quotient forgets the $e_v$-phase and keeps only the condition $\alpha^2=1$; hence this endpoint contributes two vertices indexed by $\alpha=\pm1$.
Similarly, the endpoint $2e_v$ contributes two vertices indexed by $\beta=\pm1$.
The four open edges over the interior of $E$ connect these two pairs of endpoints.
Consequently the $u,v$-phase factor is the complete bipartite graph
\[
  G=K_{2,2},
\]
which is a PL circle.

Let $\lambda_E(x)$ be the total barycentric weight of $x\in\partial\Delta'$ on the two vertices $2e_u,2e_v$.
On the locus $\lambda_E(x)>0$, write $\lambda_i(x)$ for the barycentric weights of the toric vertices $v_i$ and set
\[
  \xi(x):=-\lambda_E(x)^{-1}\sum_i \lambda_i(x)v_i\in N_\RR .
\]
If the toric vertices of $\tau_x$ lie in a cone $\sigma\in\Sigma$, then $\xi(x)\in -\sigma$, and the phase equations give $\theta\in\sigma^\perp_{\bSi}$.
Thus the toric part over $\lambda_E>0$ is
\[
  \bigcup_{\sigma\in\Sigma}\sigma^\perp_{\bSi}\times(-\sigma)=\Lambda.
\]
Together with the $u,v$-phase factor this identifies the locus $\lambda_E>0$ in $S_{\Delta,\cT,K}$ with $\Lambda\times G$.

When $\lambda_E(x)$ tends to $0$, the covector $\xi(x)$ tends radially to infinity, and the toric part approaches the Legendrian boundary $\partial\Lambda\subset S^*Q$.
At $\lambda_E(x)=0$, the point $x$ lies in the $P$-face of $\Delta$; the smallest face of $K$ containing $x$ is contained in the lineality space $N_\RR$, and the RSTZ quotient forgets both $\alpha$ and $\beta$.
Hence the whole circle $G$ collapses to one point at $\lambda_E=0$.
The $u,v$-factor at the end is therefore the cone
\[
  C(G)\cong C(S^1)\cong D^2.
\]
The punctured cone $C(G)\setminus\{\text{cone point}\}$ is identified with $G\times I$ for a collar interval $I$. After choosing a PL homeomorphism $G\cong S^1$ and identifying $C(G)$ with the closed disk $D^2$, this model gives
\[
  \fL_{\rm RSTZ}
  \cong
  (\Lambda\times S^1)
  \cup_{\partial\Lambda\times S^1\times I}
  (\partial\Lambda\times D_\disk)
  =\fL.
\]

By construction, the RSTZ image $j(S_{\Delta,\cT,K})\subset\cY$ is a strong deformation retract of $\cY$.
Via the above identification of its interior model with $\fL$, we regard $\fL$ as the embedded RSTZ skeleton in $\cY$.
\end{proof}

\subsection{A Weinstein sector $U_\Lambda\subset T^*Q$ with core $\Lambda$}
The standard cotangent Liouville form $\lambda_{\mathrm{can}}$ on $T^*Q$ has core the zero section, not the conic Lagrangian $\Lambda$.
We will construct a Weinstein sector with core $\Lambda$ from the ribbon $F$.

\smallskip

Recall from \cref{sec:skeleton-definition} that Shende constructs a Weinstein pair $(T^*Q,F)$ whose relative skeleton is the FLTZ Lagrangian $\Lambda$, and in particular $\mathrm{core}(F)=\partial\Lambda$ \cite[Prop.~16]{Shende22}.
Following the gluing discussion of Gammage--Shende \cite[\S4.1]{gammage-shende-large-volume}, we view $F$ as the replacement for the jet-bundle neighborhood used there: the hypersurface $F$ provides a smooth Weinstein neighborhood of $\partial\Lambda$ inside $\partial_\infty T^*Q=S^*Q$.

\smallskip

One can construct a Liouville sector from a sutured Liouville manifolds \cite[Lem.~2.13, Def.~2.14]{GPS1}.
Let $U_\Lambda\subset T^*Q$ denote the Liouville sector associated to the sutured Liouville manifold $(T^*Q,F)$, i.e. the completion of $T^*Q$ away from a standard neighborhood of $F$ \cite[Lem.~2.13]{GPS1}.
We equip $U_\Lambda$ with the induced Liouville form and a compatible exhausting Weinstein function, which we denote by $(\lambda_\Lambda,\phi_\Lambda)$.
By construction, $U_\Lambda$ has the conic end of the form
\begin{equation}
  \label{eq:collar-U-lambda}
  \mathrm{Nbd}(\partial U_\Lambda)\cong F\times T^*(0,\varepsilon)_t,
  \qquad
  \lambda_\Lambda=\lambda_F+p_t\,dt
\end{equation}
for $\varepsilon>0$ sufficiently small, and on $\Nbd(\partial U_\Lambda)$, $\phi_\Lambda = \phi_F + \frac{p_t^2}{2}$ where $\phi_F$ is the Weinstein function on $F$.

\begin{prop}
  \label{prop:core-U-lambda}
The Weinstein sector $(U_\Lambda,\lambda_\Lambda,\phi_\Lambda)$ has core $\mathrm{core}(U_\Lambda)=\Lambda$.
\end{prop}
\begin{proof}
For a sutured Liouville manifold $(X,F)$, the associated sector has skeleton equal to the relative skeleton of the sutured manifold, namely
\[
  \mathrm{core}(X,F):=\mathrm{core}(X)\ \cup\ \big(\text{Liouville cone on }\mathrm{core}(F)\big),
\]
see \cite[Lem.~2.13, Def.~2.14, Rem.~2.16]{GPS1}.
In our case $X=T^*Q$, so $\mathrm{core}(X)$ is the zero section, and $\mathrm{core}(F)=\partial\Lambda$ by construction.
Therefore $\mathrm{core}(U_\Lambda)=\Lambda.$
\end{proof}

\subsection{A Weinstein neighborhood $U$ of $\fL$ by Liouville gluing}\label{subsec:U_fL}
Define the left piece
\[
  U_1:=U_\Lambda\times T^*S^1,
\qquad
  \lambda_1:=\lambda_\Lambda + p_\theta d\theta,
\qquad
  \phi_1:=\phi_\Lambda + \tfrac12 p_\theta^2.
\]
By \cref{prop:core-U-lambda}, $\mathrm{core}(U_1)=\Lambda\times S^1$.

By the collar identification \eqref{eq:collar-U-lambda}, the sector $U_\Lambda$ has a standard collar near its boundary, hence after taking the product with $T^*S^1$ we obtain an exact symplectic overlap chart:
\[
  U_{12}:=F\times T^*S^1\times T^*(0,\varepsilon)
\subset U_1,
\qquad
  \lambda_1|_{U_{12}}=\lambda_F+p_\theta d\theta+p_t dt.
\]

Let $D_\disk$ be the open unit disk with the chosen collar $S^1\times I\hookrightarrow D_\disk$.
Define the right piece
\[
  U_2:=F\times T^*D_\disk,
\qquad
  \lambda_2:=\lambda_F+\lambda_{D_\disk},
\qquad
  \phi_2:=\phi_F+\phi_{D_\disk},
\]
where $\lambda_{D_\disk}$ is the canonical $1$--form on $T^*D_\disk$ and $\phi_{D_\disk}$ is any exhausting Weinstein function on $T^*D_\disk$ which, on the collar chart $T^*(S^1\times I)\cong T^*S^1\times T^*I$, agrees with $\tfrac12(p_\theta^2+p_t^2)$.
On this collar we have a canonical identification
\[
  F\times T^*(S^1\times I)\cong F\times T^*S^1\times T^*I
\]
under which
\[
  \lambda_2|_{U_{12}}=\lambda_F+p_\theta d\theta+p_t dt.
\]
Thus $U_{12}$ is an exact symplectic overlap for $(U_1,\lambda_1)$ and $(U_2,\lambda_2)$.
We define the Weinstein manifold
\begin{equation}
  \label{eq:def-U}
  U:=U_1\cup_{U_{12}}U_2,
\end{equation}
glued by the identity on $U_{12}$.
Because the primitives $\lambda_1$ and $\lambda_2$ agree on $U_{12}$, they glue to a global Liouville form $\lambda$ on $U$; by construction $\phi_1$ and $\phi_2$ may be chosen to agree on $U_{12}$, hence glue to a global Weinstein function $\phi$.

\begin{prop}
  \label{prop:core-U}
The Weinstein manifold $U$ defined in \eqref{eq:def-U} has core
\[
  \mathrm{core}(U)=\fL.
\]
\end{prop}
\begin{proof}
We have $\mathrm{core}(U_1)=\Lambda\times S^1$ and $\mathrm{core}(U_2)=\mathrm{core}(F)\times D_\disk=\partial\Lambda\times D_\disk$.
On the overlap $U_{12}$, the core is $\partial\Lambda\times S^1\times I$.
Since the Liouville vector fields agree on $U_{12}$, the core of the glued sector is the union of the cores glued along their common part, i.e.
\[
  \mathrm{core}(U)=(\Lambda\times S^1)\cup_{\partial\Lambda\times S^1\times I}(\partial\Lambda\times D_\disk)=\fL.
\]
\end{proof}

\subsection{Brane structures and grading data}\label{subsec:polarization}
To speak about $\bZ$-graded microlocal sheaves, we must fix brane data: a stable polarization (equivalently, grading/Maslov data) together with a $\bZ/2$ background class; see \cite[\S2]{GPS3} and \cite[\S2]{nadlershende20}.
In this paper we work in the \emph{stably polarized} framework of \cref{def:mush-weinstein}, following \cite{shende2021,nadlershende20,GPS3}.

\smallskip\noindent
\emph{Polarization on $U$ from cotangent models.}
Recall that $U$ was constructed by Liouville gluing $U=U_1\cup_{U_{12}}U_2$ with
\[
  U_1=U_\Lambda\times T^*S^1,\qquad
  U_2=F\times T^*D_\disk,\qquad
  U_{12}=F\times T^*S^1\times T^*(0,\varepsilon).
\]
Each cotangent bundle $T^*M$ carries a canonical Lagrangian distribution (polarization) given by the vertical tangent subbundle
\[
  T_{\mathrm{vert}}(T^*M):=\ker(d\pi_M)\subset T(T^*M),
\]
which is Lagrangian for the canonical symplectic form.

We also define a canonical polarization on the Weinstein hypersurface $F\subset S^*Q$ as follows.
Let $\pi:S^*Q\to Q$ be the projection and let $\alpha$ be the restriction of the canonical $1$--form on $T^*Q$ (so $\ker(\alpha)$ is the standard contact distribution on $S^*Q$).
Set $\xi:=\ker(\alpha)$, which is symplectic for $d\alpha$, and consider the vertical subbundle
\[
  V:=\ker(d\pi)\cap \xi\subset \xi.
\]
Since $d\alpha$ pairs vertical directions only with horizontal ones, $V$ is isotropic; moreover $\mathrm{rank}(V)=n-1=\tfrac12\mathrm{rank}(\xi)$, hence $V$ is Lagrangian inside $\xi$.
Because $F$ is a Weinstein hypersurface, the form $d(\alpha|_F)=d\alpha|_{TF}$ is symplectic; in particular, the Reeb direction of $\alpha$ is transverse to $F$, and the projection $TF\to \xi|_F$ along the Reeb field is a symplectic isomorphism.
We define $T_{\mathrm{vert}}F\subset TF$ to be the inverse image of $V|_F$ under this identification; it is a Lagrangian distribution on $(F,d\lambda_F)$.

Next, we choose a polarization on the sector $U_\Lambda$.
On the interior of $U_\Lambda$ (i.e.\ away from the boundary at infinity), we view $U_\Lambda$ as an open subset of $T^*Q$ and take the restriction of the cotangent polarization $T_{\mathrm{vert}}(T^*Q)=\ker(d\pi_Q)$.
Near the boundary, we use the collar chart \eqref{eq:collar-U-lambda}
\[
  \mathrm{Nbd}(\partial U_\Lambda)\cong F\times T^*(0,\varepsilon)_t.
\]
In standard cotangent coordinates $(\theta,p)$ on $T^*Q$, the cotangent polarization is spanned by the fiber vectors $\partial_{p_i}$.
Passing to polar coordinates $p=r\cdot u$ on the fibers (where $u\in S^*_qQ$ and $r>0$), this polarization splits as the direct sum of the vertical directions tangent to the unit sphere (giving $V$ above) and the radial direction $\partial_r$.
Restricting to $F\subset S^*Q$ and identifying the collar parameter $t$ with a monotone function of $r$, we obtain the product polarization
\[
  \ker(d\pi_Q)|_{\mathrm{Nbd}(\partial U_\Lambda)\subset U_\Lambda}=T_{\mathrm{vert}}F\oplus T_{\mathrm{horiz}}\bigl(T^*(0,\varepsilon)\bigr)\subset T\bigl(F\times T^*(0,\varepsilon)\bigr)
\]
on the collar, where $T_{\mathrm{horiz}}(T^*(0,\varepsilon))$ is generated by $\partial_t$.
We take
\[
  T_{\mathrm{pol}}U_\Lambda:=
  \begin{cases}
  \ker(d\pi_Q)|_{U_\Lambda}\subset TU_\Lambda,\text{away from $\mathrm{Nbd}(\partial U_\Lambda)$, or $\varepsilon_2<t<\varepsilon$ on $\mathrm{Nbd}(\partial U_\Lambda)$};\\
  T_{\mathrm{vert}}F\oplus\mathbb R\left (\frac{t-\varepsilon_1}{\varepsilon_2-\varepsilon_1} \partial_t+ \frac{\varepsilon_2-t}{\varepsilon_2-\varepsilon_1} \partial_p\right ), \text{$\varepsilon_1\leq t\leq\varepsilon_2$ on $\mathrm{Nbd}(\partial U_\Lambda)$};\\
  T_{\mathrm{vert}} F \oplus T_{\mathrm{vert}}(T^*(0,\varepsilon)), \text{$t<\varepsilon_1$ on $\mathrm{Nbd}(\partial U_\Lambda)$.} 
  \end{cases}
  \]
Recall that the collar of $U_2$ is
\[
\mathrm{Nbd}(\partial U_2)=F\times T^*(S^1\times (0,\varepsilon))\hookrightarrow U_2=F\times T^*D_{\mathrm{disk}}.
\]
Finally we set
\begin{align*}
  &T_{\mathrm{pol}}U_1:=T_{\mathrm{pol}}U_\Lambda\oplus T_{\mathrm{vert}}(T^*S^1)\subset TU_1,\\
  &T_{\mathrm{pol}}U_2:=
  \begin{cases}
    T_{\mathrm{vert}}F\oplus T_{\mathrm{vert}}(T^*D_\disk)\subset TU_2, \text{away from $\mathrm{Nbd}(\partial U_2$), or $t<\varepsilon_1$ on $\mathrm{Nbd}(\partial U_2)$};\\
    T_{\mathrm{vert}}F\oplus\mathbb R\left (\frac{t-\varepsilon_1}{\varepsilon_2-\varepsilon_1} \partial_t+ \frac{\varepsilon_2-t}{\varepsilon_2-\varepsilon_1} \partial_p\right )\oplus T_{\mathrm{vert}}(T^*S^1), \text{$\varepsilon_1\leq t\leq\varepsilon_2$ on $\mathrm{Nbd}(\partial U_2)$};\\
    T_{\mathrm{vert}} F\oplus T_{\mathrm{horiz}}(T^*(0,\varepsilon))\oplus T_{\mathrm{vert}}(T^*S^1), \text{$t>\varepsilon_2$ on $\mathrm{Nbd}(\partial U_2)$.}
    \end{cases}
\end{align*}
By construction, they coincide on $U_{12}$, and hence glue to a Lagrangian distribution $T_{\mathrm{pol}}U\subset TU$.
In particular, it is a polarization of $U$; we denote the corresponding section of the Lagrangian Grassmannian bundle by
\[
  \mathfrak{p}:U\longrightarrow \mathrm{LGr}(TU).
\]

\subsection{Descent for $\mush$ on the core $\fL$}
\begin{prop}
  \label{prop:A-sheaf-side-colim}
The wrapped microlocal category $\mush^w_\fL(\fL)$ fits into the following homotopy pushout diagram.
\begin{equation}
  \label{eqn:A-sheaf-side-colim}
  \begin{tikzcd}
      \mush^w_\fL(\partial \Lambda \times S^1\times I) \arrow[d] \arrow[r] & \mush^w_\fL(\partial\Lambda\times D_\disk) \arrow[d] \\
      \mush^w_\fL(\Lambda\times S^1) \arrow[r]  & \mush^w_\fL(\fL).
  \end{tikzcd}
\end{equation}
\end{prop}
\begin{proof}
    Apply \cref{rem:sheaf_cosheaf} to the cover $\fL=(\Lambda\times S^1)\cup(\partial\Lambda\times D_\disk)$ with intersection $\partial\Lambda\times S^1\times I$.
\end{proof}

\section{Affine Blow-up as push-out}
\label{sec:b-model}
The main result in this section is \cref{B-side-colim}, which generalizes \cite[Thm.~4.13]{gammage2022mirror}.

Let $S$ be a $k$-variety, and $P$ be an effective Cartier divisor in $S$. Define $X=S\times \GG_m$, $D= P\times \GG_m$ and $H=P\times \{1\} \subset D$. We consider the blow up $\Bl_{H} X$.

\begin{prop} \label{aff-blow-up}
The blow-up of $X$ along $H$ minus the strict transform of $D$
\[
\Bl_{H} X \setminus \tD
\]
is isomorphic to $\Tot(\cO(-P)) \setminus \{w\mid \langle w, v \rangle+1=0\}$. Here $v$ is the canonical section of $\cO(P)$, and $w\in \Tot(\cO(-P))$ takes fiberwise pairing with $v$.
\end{prop}
\begin{proof}

Let $L=\cO(P)$, and consider the rank $2$ vector bundle $\cE:=\cO_S\oplus L$ on $X$ and its projectivization $\bP(\cE)\to S$. Over a point $x\in S$, write homogeneous coordinates $[U:V]$ on the fiber $\bP(\cE_x)$, with $U$ corresponding to the $\cO_X$ summand and $V$ to the $L$ summand. The blow–up $\Bl_H X$ is realized as the closed subvariety
\[
\Big\{((x,t),[U:V])\in (S\times \GG_m)\times_S \bP(\cE) : (t-1)V=v(x)U\Big\}.
\]

On the open locus $t\neq 1$, $v=0$ is equivalent to $V=0$.
Hence the strict transform $\widetilde D=\overline{\pi^{-1}(D \setminus H)}$ is the Cartier divisor $\{V=0\}\subset \Bl_H X$.

Therefore
\[
   \Bl_H X \setminus \tD=\{V\neq 0\}.
\]

On this chart set $w:=U/V$, and fiberwisely $w$ is a
coordinate in $L^{-1}$. The incidence equation becomes $
t-1=\langle w,v\rangle .$ This induces an embedding map $\Phi$ as
\begin{align*}
   \Phi:\ \Bl_H X\ &\longrightarrow\ \Tot(L^{-1}),\\
   ((x,t),[U:V]) &\longmapsto (x,w=U/V),
\end{align*}
whose image is cut out by the condition $\langle w,v\rangle +1\neq 0$.
\end{proof}

\begin{defi}
    Let $\cC,\cC_\pm$ be stable categories, and $i_\pm:\cC_\pm\to\cC$ be fully faithful exact functors. We say $(i_+, i_-)$ is a (2-term) {\bf semi-orthogonal decomposition} of $\cC$ if the following conditions hold:\begin{enumerate}
        \item $\Hom(i_-c_-, i_+c_+)=0$ for any $c_\pm\in\cC_\pm$;
        \item For any $c\in\cC$, there is a fiber sequence $$i_-c_-\to c\to i_+c_+,$$
        where $c_\pm\in\cC_\pm$.
    \end{enumerate}
\end{defi}

\begin{prop}\label{fracture}
    Under the conditions above, we have:\begin{enumerate}
        \item $i_-$ admits a right adjoint, $i_+$ admits a left adjoint, and there is a fiber sequence of endo-functors:
            $$i_-i_-^\R\to\id\to i_+i_+^\L.$$
        \item If $i_+$ admits a right adjoint, then $(i_-^{\R\R},i_+)$ is a semi-orthogonal decomposition of $\cC$. There is a fiber sequence $i_+i_+^\R\to\id\to i_-^{\R\R}i_-^\R$ consequently.
    \end{enumerate}
\end{prop}
\begin{proof}
    The first statement is trivial, and proves a semi-orthogonal decomposition is equivalent to a left split Verdier sequence $\cC_+\xrightarrow{i_+}\cC\xrightarrow{i_-^\R}\cC_-$ by \cite[Cor.~A.2.17]{hermitian}. If $i_+$ admits a right adjoint, then $\cC_-\xrightarrow{i_-^{\R\R}}\cC\xrightarrow{i_+^\R}\cC_+$ is also a left split Verdier sequence by \cite[Lem.~A.2.8]{hermitian}.
\end{proof}
The following proposition is a direct corollary of \cite[Thm.~B.8]{Lax_Additivity}. Here we present another proof.
\begin{prop}\label{SOD_pushout}
    In the diagram below, suppose the following statements hold:\begin{enumerate}
        \item $(i_+,i_-)$ is a semi-orthogonal decomposition of $\cC$.
        \item The functors $i_+,f_-$ admit right adjoints.
        \item There is an equivalence $i_+^\R i_-=f_+f_-^\R$, defining $\varphi:i_+f_+\to i_-f_-$.
    \end{enumerate}
    Then coequalizer of $\varphi$ coincides with the pushout of $\cC_\pm$ over $\mathcal{A}$.
\[
\begin{tikzcd}[row sep=1.2em, column sep=3.2em]
 & \cC_+ \arrow[rd, "i_+"] \arrow[dd, "\varphi", Rightarrow] &\\
\mathcal{A} \arrow[rd, "f_-"'] \arrow[ru, "f_+"] & & \cC \\
& \cC_- \arrow[ru, "i_-"']  &    
\end{tikzcd}
\]

\end{prop}
\begin{proof}
    Arbitrarily pick a stable category $\cD$. By definition, we must find an equivalence between two spaces $$\{u\in\Fun^{\ex}(\cC,\cD)\mid u\varphi\mbox{ is invertible}\};$$
    $$\{(u_\pm\in\Fun^\ex(\cC_\pm,\cD),\varepsilon:u_+f_+\to u_-f_-)\mid \varepsilon\mbox{ is invertible}\}.$$
    In one direction, we construct the map as $$u\mapsto(u_\pm=ui_\pm,\varepsilon=u\varphi).$$
    To construct the inverse, we apply that the composite 
    $$i_+i_+^\L[-1]\to i_+i_+^\R i_+i_+^\L[-1]\to i_+i_+^\R i_-i_-^\R=i_+f_+f_-^\R i_-^\R\xrightarrow{\varphi} i_-f_-f_-^\R i_-^\R\to i_-i_-^\R$$ is identified with the natural map in \cref{fracture} by diagram chasing, hence the cofiber is $\id_\cC$. If a tuple $(u_\pm,\varepsilon)$ is given, we construct
    $$u=\cofib[u_+i_+^\L[-1]\to u_+i_+^\R i_+i_+^\L[-1]\to u_+i_+^\R i_-i_-^\R\to u_+f_+f_-^\R i_-^\R\xrightarrow{\varepsilon} u_-f_-f_-^\R i_-^\R\to u_-i_-^\R].$$
    Identifying $ui_\pm=u_\pm$, it suffices to show $u\varphi=\varepsilon$.
    In the formula $$u\varphi=[u_+f_+\to u_+i_+^\R i_+f_+\to u_+i_+^\R i_-f_-\to u_+f_+f_-^\R f_-\to u_-f_-f_-^\R f_-\to u_-f_-],$$
    we unwind the definition of the second arrow as 
    $$u_+i_+^\R\varphi=u_+i_+^\R[i_+f_+\to i_+f_+f_-^\R f_-\to i_+ i_+^\R i_-f_-\to i_-f_-].$$
    A similar diagram chasing simplifies it into $u\varphi=\varepsilon$ as desired.
\end{proof}

\begin{prop}\label{lemma:qcqs}
    Let $X$ be a qcqs scheme over $k$. Let $H\subset X$ be an lci closed subscheme of codimension 2. Let $D\subset X$ be an effective Cartier divisor containing $H$, as the former diagram below, where $E$ is the exceptional divisor. Set $\Phi=r_*(q^*(-)\otimes\cO_E(-1))$ and $k_\times:=(k^*)^\L=k_*(-)\otimes\cO_X(D)[-1]$. Then the latter diagram below satisfies the conditions of \cref{SOD_pushout}.
\\ \noindent\begin{minipage}{0.5\linewidth}
\[
\begin{tikzcd}
E \arrow[d, "q"] \arrow[rr, "r"] &                  & \Bl_H X \arrow[d, "p"] \\
H \arrow[r, "i"]                 & D \arrow[r, "k"] & X                     
\end{tikzcd}
\]
\end{minipage}
\begin{minipage}{0.5\linewidth}
\[
\begin{tikzcd}[row sep=1.2em, column sep=3.2em]
 & \Coh(H) \arrow[rd, "\Phi"] \arrow[dd, "\varphi", Rightarrow] &\\
\Coh(D) \arrow[rd, "k_\times"] \arrow[ru, "i^*"'] & & \Coh(\Bl_H X) \\
& \Coh(X) \arrow[ru, "{p^*[1]}"']  &    
\end{tikzcd}
\]
\end{minipage}
\end{prop}

\begin{proof}
    The semi-orthogonal decomposition is proved in \cite[Thm.~3.4]{SOD_kuz}; the right adjoints are known as Grothendieck duality. For the last condition, we have:
    $$\Phi^\R=q_*(r^\times(-)\otimes\cO_E(1))=q_*(r^*(-)\otimes r^*\cO_{\Bl_H X}(E)\otimes\cO_E(1))[-1]=q_*r^*[-1],$$
    $$\Phi^\R p^*[1]=q_*r^*p^*=q_*q^*i^*k^*=i^*k^*(-)\otimes q_*\cO_E=i^*k^*.$$
\end{proof}

\begin{theorem} \label{B-side-colim}
    In the situation of \cref{lemma:qcqs}, we suppose that $X$ is Noetherian, and $H$ is a transverse intersection of $D$ and another effective Cartier divisor. Then there is a pushout diagram:\\
    \begin{equation}
        \label{eqn:B-side-colim}
        \begin{tikzcd}
            \Coh(D) \arrow[d, "k_\times"'] \arrow[r, "i^*"] & \Coh(H) \arrow[d] \\
            \Coh(X) \arrow[r]                          & \Coh(\Bl_H X\setminus\tD)
        \end{tikzcd}
    \end{equation}
\end{theorem}

\begin{proof}
    
    Let $T=p^*D\subset \Bl_H X$ and name the morphisms as the following diagram. Note that the right square is a tor-independent base change.
\[
  \begin{tikzcd}
                                & \tilde{D} \arrow[d, "b"]          &                        \\
E \arrow[d, "q"] \arrow[r, "a"] & T \arrow[d, "\pi"] \arrow[r, "t"] & \Bl_H X \arrow[d, "p"] \\
H \arrow[r, "i"]                & D \arrow[r, "k"]                  & X.                     
\end{tikzcd}\]
  We may omit the isomorphism $\pi b=p|_{\tilde{D}}:\tilde{D}\congto D$ along the proof. 

    Following \cref{lemma:qcqs}, it suffices to find $\cofib~\varphi$.
    We start with the short exact sequence \[0\to a_*\cO_E\otimes t^*\cO_{\Bl_H X}(-\tilde{D})\to \cO_T\to b_*\cO_{\tilde{D}}\to0.\]
    Let $\cF\in\Coh(D)$. Then we apply $t_*(\pi^*\cF\otimes t^*\cO_{\Bl_H X}(T)\otimes-)$ to this sequence to obtain a fiber sequence.

    The first term is \[t_*(\pi^*\cF\otimes a_*\cO_E(-1))=(ta)_*((\pi a)^*\cF\otimes a_*\cO_E(-1))=\Phi i^*\cF.\]

    The second term is \[t_*(\pi^*\cF\otimes t^*\cO_{\Bl_H X}(T))=t_*\pi^*\cF\otimes\cO_{\Bl_H X}(T)=p^*k_*\cF\otimes p^*\cO_X(D)=p^*k_\times\cF[1].\]

    The third term is \[t_*(\pi^* F\otimes t^*\cO_{\Bl_H X}(T)\otimes b_*\cO_{\tilde{D}})=\cO_{\Bl_H X}(T)\otimes(tb)_*\cF.\]

    It suffices to give a homotopy connecting the first arrow and $\varphi_\cF$. Locally everything is flat base change of $*\hookrightarrow \AA^1\hookrightarrow\AA^2$. We omit the proof for this universal example.
\end{proof}

\section{Categorical equivalences}\label{sec:cat_eq}

Let $S$ be the Fano toric orbifold with stacky fan $\bSi$ fixed in \cref{sec:setup}, and let $P\subset S$ denote its toric boundary divisor.
Let $K_S$ be the canonical bundle of $S$, with stacky fan $\wt\Sigma\subset N_\bR\oplus\bR$ and ray vectors $\{(v,1)\mid v\in\boldb\}\subset N\times\bZ$, and let $w:K_S\to \C$ be the toric function corresponding to the character $(\vec 0,1)$.
Set
\[
  \cX^\circ:=K_S\setminus\{w+1=0\}.
\]
By \cref{aff-blow-up}, we may identify $\cX^\circ\cong \Bl_H X\setminus\tD$, where
\[
  X:=S\times \GG_m,\qquad D:=P\times \GG_m,\qquad H:=P\times\{1\}\subset D,
\]
and $\tD$ denotes the strict transform of $D$.

\smallskip

Let $\Lambda=\Lambda_\bSi\subset T^*Q$ be the (stacky) FLTZ skeleton from \cref{sec:skeleton-definition}. Recall the three pieces of the skeleton $\fL$ from \cref{subsec:U_fL}:
\[
  \fL_1:=\Lambda\times S^1,\qquad \fL_{12}:=\partial\Lambda\times S^1\times I,\qquad \fL_2:=\partial\Lambda\times D_{\mathrm{disk}},
\]
where $D_{\mathrm{disk}}$ denotes the open unit disk introduced, which agrees with  $\partial\Lambda\times D_\disk$ appearing in \eqref{eqn:A-sheaf-side-colim}.

\begin{lemm}\label{prop:identify-corners}
There are equivalences of stable categories
\[
  \mush^w_\fL(\fL_1)\ \simeq\ \Coh(X),\qquad
  \mush^w_\fL(\fL_{12})\ \simeq\ \Coh(D),\qquad
  \mush^w_\fL(\fL_2)\ \simeq\ \Coh(H).
\]
\end{lemm}

\begin{proof}
We have an equivalence of large categories
\[
  \mush_\Lambda(\Lambda)\ \simeq\ \Sh_{\Lambda}(Q)
\]
(\cite{Nad16}, \cite{nadlershende20} and \cite{GPS3}).
Taking compact objects yields
\[
  \mush^w_{\Lambda}(\Lambda)\ \simeq\ \Sh^{w}_{\Lambda}(Q).
\]
Non-equivariant coherent-constructible correspondence for orbifolds (\cite{kuwagaki20}) identifies $\Sh^{w}_{\Lambda}(Q)$ with $\Coh(S)$; hence
\begin{equation}\label{eq:CCC-S}
  \mush^w_{\Lambda}(\Lambda)\ \simeq\ \Coh(S).
\end{equation}

Similarly, Gammage--Shende prove (\cite[Thm.~7.13]{shende-gammage} and \cite[Rem.~7.14]{shende-gammage} for variants beyond the smooth case)
\begin{equation}\label{eq:CCC-boundary}
  \mush^w_{\partial\Lambda}(\partial\Lambda)\ \simeq\ \Coh(P).
\end{equation}
Then K\"unneth property for microlocal sheaves from \cite{kuo2025dualitykernelsmicrolocalgeometry} gives
\[
  \mush^w_\fL(\Lambda\times S^1)\ \simeq\ \mush^w_{\Lambda}(\Lambda)\otimes \mush^w_{S^1}(S^1),
\qquad
  \mush^w_\fL(\partial\Lambda\times S^1\times I)\ \simeq\ \mush^w_{\partial\Lambda}(\partial \Lambda)\otimes \mush^w_{S^1}(S^1),
\]
since $I$ is contractible.
Moreover, $\mush^w_{S^1}(S^1)\simeq \Perf(k[t,t^{-1}])=\Coh(\GG_m)$, so using \eqref{eq:CCC-S} and \eqref{eq:CCC-boundary} we obtain
\[
  \mush^w_\fL(\fL_1)\ \simeq\ \Coh(S)\otimes \Coh(\GG_m)\ \simeq\ \Coh(S\times\GG_m)=\Coh(X),
\]
and likewise $\mush^w_\fL(\fL_{12})\simeq \Coh(P\times\GG_m)=\Coh(D)$.

Finally, since $D_{\mathrm{disk}}$ is contractible, $\mush^w_{D_{\mathrm{disk}}}(D_\mathrm{disk})\simeq \Perf(k)=\Coh(\Spec k)$, and the same K\"unneth argument gives
\[
  \mush^w_\fL(\fL_2)\ \simeq\ \mush^w_{\partial\Lambda}(\partial \Lambda)\otimes \mush^w_{D_{\mathrm{disk}}}(D_\mathrm{disk})\ \simeq\ \Coh(P)\otimes \Coh(\Spec k)\ \simeq\ \Coh(P)=\Coh(H),
\]
since $H=P\times\{1\}\cong P$.
\end{proof}

\smallskip

\begin{lemm}\label{prop:compare-arrows}
Under the identifications in \cref{prop:identify-corners}, the upper horizontal arrow and left vertical arrow in \eqref{eqn:A-sheaf-side-colim} identify with the upper horizontal arrow and left vertical arrow in \eqref{eqn:B-side-colim}, namely
\[
  i^*:\Coh(D)\to \Coh(H),
\qquad\text{and}\qquad
  k_\times:\Coh(D)\to \Coh(X).
\]
\end{lemm}

\begin{proof}
The left vertical arrow in \eqref{eqn:A-sheaf-side-colim} is the corestriction functor for the inclusion $\fL_{12}\subset\fL_1$.
It suffices to treat the $\Lambda$-factor and then tensor with $\mush^w_{S^1}(S^1)\simeq \Coh(\GG_m)$.

On the B-side, the inclusion $k:D=P\times\GG_m\hookrightarrow X=S\times\GG_m$ is induced from the inclusion $i_P:P\hookrightarrow S$ of the toric boundary divisor.
Gammage--Shende identify the pullback $i_P^*:\Coh(S)\to\Coh(P)$ with a microlocalization functor on the A-side (``restriction is mirror to microlocalization''), by comparing both functors on the standard toric generators; see \cite[\S7.2, esp. Lem.~7.4 and Lem.~7.5]{shende-gammage}.
Taking left adjoints, and using that the wrapped microlocal categories form a cosheaf whose corestriction maps arise as these adjoints (as explained in the proof of \cite[Thm.~7.13]{shende-gammage}), we obtain that the corestriction functor
\[
  \mush^w_{\partial\Lambda}(\partial \Lambda)=\mush^w_\Lambda(\partial \Lambda)\longrightarrow \mush^w_{\Lambda}(\Lambda).
\]
corresponds to the functor $(i_{P}^*)^\L:\Coh(P)\to\Coh(S)$.
Tensoring with $\Coh(\GG_m)$ then identifies the corestriction $\mush^w_\fL(\fL_{12})\to \mush^w_\fL(\fL_1)$ with
\[
  k_\times=(i_{P}^*)^\L\otimes \id_{\Coh(\GG_m)}:\ \Coh(P\times\GG_m)\to \Coh(S\times\GG_m).
\]

\smallskip\noindent
Under \cref{prop:identify-corners}, the upper horizontal arrow in \eqref{eqn:A-sheaf-side-colim} is induced by the inclusion of the collar $S^1\times I\hookrightarrow D_{\mathrm{disk}}$ on the $\mush^w_{S^1}(S^1)$-factor, tensored with $\Coh(P)$.
Identifying $\mush^w_{S^1}(S^1)\simeq \Perf(k[t,t^{-1}])$ and $\mush^w_{D_{\mathrm{disk}}}(D_\mathrm{disk})\simeq \Perf(k)$, this functor is the derived tensor product
\[
  -\otimes_{k[t,t^{-1}]} k,
\]
which is the standard algebraic model for pullback along the inclusion $\{1\}\hookrightarrow \GG_m$, hence agrees with
\[
  i^*:\Coh(P\times\GG_m)\to \Coh(P\times\{1\}).
\]
\end{proof}

\smallskip

 By \'{e}tale descent, \cref{B-side-colim} also holds on DM stacks. Together with \cref{prop:A-sheaf-side-colim}, both \eqref{eqn:A-sheaf-side-colim} and \eqref{eqn:B-side-colim} are colimit diagrams in $\catperf$ (for \eqref{eqn:B-side-colim}, note that $\catperf\to\widehat{\catex}$ reflects pushouts). By \cref{prop:identify-corners} and \cref{prop:compare-arrows}, they are equivalent as colimit diagrams of equivalent data. In particular, we have
\[
  \mush^w_\fL(\fL)\ \simeq\ \Coh(\Bl_H X\setminus\tD)\ \simeq\ \Coh(\cX^\circ).
\]
This completes the categorical equivalence,
and finishes the proof of \cref{thm:main}. We denote this equivalence
$\Coh(\cX^\circ)\congto\mush^w_\fL(\fL)$ by $\kappa_{\cX^\circ}$.

\section{Characteristic cycles for finite microlocal-rank objects}
\label{sec:characteristic-cycles}

In this section we study the characteristic cycle for the sheaf $\mush^c$ of finite microlocal-rank objects. Every functor associated to the (Lagrangian) cycle sheaf and the orientation sheaves is underived.

\subsection{Global ambient construction}
\label{subsec:cc-global-ambient}

Let $(W,\lambda,\phi)$ be a stably polarized Weinstein manifold, and let $\Lambda\subset W$ be a compact subanalytic Lagrangian subset.
We assume that $\Lambda$ satisfies the following two conditions.
\begin{enumerate}
    \item The set $\Lambda$ admits a finite Whitney stratification
    \[
      \Lambda=\bigsqcup_{\alpha\in A}\Lambda_\alpha
    \]
    by connected smooth isotropic strata, and every top-dimensional stratum is Lagrangian.
    \item There exists a closed Legendrian subset
    \[
      \widetilde\Lambda\subset W\times \RR_z
    \]
    for the contact form $dz+\lambda$ whose projection to $W$ restricts to a homeomorphism $\widetilde\Lambda\to\Lambda$.
\end{enumerate}
Fix a stable Lagrangian-normal datum $\nu$ for $\Lambda$ in the sense of \cref{def:mush-weinstein}.
Choose an auxiliary contact embedding
\[
  \iota:W\times \RR_z\hookrightarrow S^*Z
\]
for some real analytic manifold $Z$.
Let $\Lambda^\infty\subset S^*Z$ be the image of $\widetilde\Lambda$, let
\[
  \Lambda^\infty_\nu\subset S^*Z
\]
be the Legendrian thickening determined by $\nu$, and let
\[
  i_{\Lambda,\nu}:\Lambda\hookrightarrow \Lambda^\infty_\nu
\]
be the inclusion of the original skeleton as the zero section of the thickening.
Write
\[
  q:\dot T^*Z\to S^*Z,
  \qquad
  \widehat\Lambda_\nu:=q^{-1}(\Lambda^\infty_\nu)\subset \dot T^*Z.
\]
Recall from Kashiwara--Schapira \cite[Def.~9.3.1]{KS90} that $\mathcal L_Z$ denotes the sheaf of Lagrangian cycles on $T^*Z$.
We define a sheaf of abelian groups on $\Lambda$ by
\[
  \mathcal Z_{\Lambda,\nu}:=i_{\Lambda,\nu}^{-1}q_*\underline{\Gamma}_{\widehat\Lambda_\nu}(\mathcal L_Z).
\]
Thus for an open subset $V\subset \Lambda$,
\[
  \Gamma(V,\mathcal Z_{\Lambda,\nu})
  =
  \varinjlim_{\substack{\Omega\subset \dot T^*Z\text{ conic open}\\ \Omega\supset q^{-1}(i_{\Lambda,\nu}(V))}}
  \Gamma_{\widehat\Lambda_\nu\cap \Omega}(\Omega,\mathcal L_Z),
\]
where the colimit is taken with respect to restriction under shrinking neighborhoods.

Let $\Lambda\subset W$ satisfy the above assumptions, with chosen stable Lagrangian-normal datum $\nu$.
Then the assignment
\[
  V\longmapsto K_0\bigl(\mush^c_\Lambda(V)\bigr)
\]
defines a presheaf of abelian groups on $\Lambda$.

\begin{defi}\label{thm:cc-global-ambient}
We define a morphism of presheaves
\[
  \CC_{\Lambda,\nu}:K_0^{\mathrm{pre}}(\mush^c_\Lambda)\longrightarrow \mathcal Z_{\Lambda,\nu}
\]
as follows.
Let $V\subset \Lambda$ be open and let $\cF\in \mush^c_\Lambda(V)$.
Since $\mush^c_\Lambda$ is obtained from the Nadler--Shende ambient cotangent model by sheafification, there is an open cover $V=\bigcup_i V_i$ such that each restriction $\cF|_{V_i}$ is represented in the ambient model by a constructible sheaf $G_i\in \Sh^c(Z)$ on a conic open neighborhood
\[
  \Omega_i\subset \dot T^*Z,
  \qquad
  \Omega_i\supset q^{-1}(i_{\Lambda,\nu}(V_i)),
\]
with
\[
  \ss(G_i)\cap \Omega_i\subset \widehat\Lambda_\nu.
\]
On $V_i$ we define $\CC_{\Lambda,\nu}([\cF])$ to be the germ of the Kashiwara--Schapira characteristic cycle
\[
  \CC(G_i)\in
  \Gamma_{\widehat\Lambda_\nu\cap \Omega_i}(\Omega_i,\mathcal L_Z),
\]
see \cite[Def.~9.4.1]{KS90}.
These local germs glue, and the resulting glued section is denoted
\[
  \CC_{\Lambda,\nu}([\cF])\in \Gamma(V,\mathcal Z_{\Lambda,\nu}).
\]
\end{defi}

We verify that the construction is well-defined.
Let $G_1$ and $G_2$ be two representatives of the same microlocal object on a conic neighborhood of $q^{-1}(i_{\Lambda,\nu}(V'))$ for some open set $V'\subset V$.
In the Verdier-quotient description of microlocalization, after shrinking to a smaller conic neighborhood $\Omega'$ of $q^{-1}(i_{\Lambda,\nu}(V'))$, the two representatives are connected by a zigzag whose cones have singular support disjoint from $\Omega'$.
Additivity of the classical characteristic cycle and the inclusion $\operatorname{supp}(\CC(-))\subset \ss(-)$ therefore imply
\[
  \CC(G_1)|_{\Omega'}=\CC(G_2)|_{\Omega'}
\]
in $\Gamma_{\widehat\Lambda_\nu\cap \Omega'}(\Omega',\mathcal L_Z)$.
Hence the local germs agree on overlaps and are independent of the chosen representatives.
The additivity of classical characteristic cycles in distinguished triangles implies that the assignment factors through $K_0$ and is compatible with restriction in $V$.

We next identify the coefficient local system on the smooth top-dimensional branches.
For a real vector bundle $E\to X$, the notation $\ori_E$ means the local system on $X$ of the orientations of fibers of $E$. It does not mean the absolute orientation sheaf of the total space of $E$.
Thus, for instance, $\ori_{D_\nu}$ below is a local system on $\Lambda_\alpha$ after identifying $\widetilde\Lambda_\alpha$ with $\Lambda_\alpha$.

Let
\[
  \Lambda^{\mathrm{top}}=\bigsqcup_{\alpha\in A_{\mathrm{top}}}\Lambda_\alpha
\]
be the union of the connected top-dimensional strata of $\Lambda$.
For $\alpha\in A_{\mathrm{top}}$, set
\[
  \mathfrak o_{\Lambda_\alpha,\nu}:=\mathcal Z_{\Lambda,\nu}|_{\Lambda_\alpha}.
\]
Let $D_\nu\subset \nu|_{\widetilde\Lambda_\alpha}$ be the auxiliary Lagrangian subbundle used in the $\nu$-thickening, and let
\[
  B_\alpha:=q^{-1}\bigl(D_\nu|_{\widetilde\Lambda_\alpha}\bigr)\subset \dot T^*Z,
\]
with projection $p_\alpha:B_\alpha\to \Lambda_\alpha$ and inclusion $i_{B_\alpha}:B_\alpha\hookrightarrow \dot T^*Z$.
The fiber of $p_\alpha$ is diffeomorphic to $\RR_{>0}\times \RR^k$, where $k=\rank D_\nu$.
The factor $\RR_{>0}$ is the positive conic radial direction and has its standard orientation.
Since the vertical tangent bundle of $D_\nu\to \Lambda_\alpha$ is the pullback of $D_\nu$, one obtains
\[
  \ori_{B_\alpha}\simeq
  p_\alpha^{-1}\bigl(\ori_{\Lambda_\alpha}\otimes \ori_{D_\nu}\bigr).
\]

Let $\pi_Z:T^*Z\to Z$ be the projection and put
\[
  T(T^*Z/Z):=\ker(d\pi_Z).
\]
Kashiwara--Schapira's description of $\mathcal L_Z$ on a smooth conic Lagrangian branch gives
\[
  p_\alpha^{-1}\mathfrak o_{\Lambda_\alpha,\nu}
  \simeq
  i_{B_\alpha}^{-1}\underline{\Gamma}_{B_\alpha}(\mathcal L_Z)
  \simeq
  \ori_{B_\alpha}\otimes \ori_{T^*Z/Z}|_{B_\alpha}.
\]
The polarized $\nu$-realization includes the following compatibility along $B_\alpha$.
Let $\RR_\rho\subset T(T^*Z/Z)|_{B_\alpha}$ be the radial line, oriented by the positive radial coordinate.
If $P$ denotes the Lagrangian distribution representing the chosen stable polarization of $W$, then there is a stable identification
\[
  T(T^*Z/Z)|_{B_\alpha}/\RR_\rho
  \simeq_{\mathrm{st}}
  p_\alpha^{-1}\bigl(P|_{\Lambda_\alpha}\oplus D_\nu^\vee\bigr).
\]
Taking determinant orientation local systems and using the canonical orientation of $\RR_\rho$ yields
\[
  \ori_{T^*Z/Z}|_{B_\alpha}
  \simeq
  p_\alpha^{-1}\bigl(\ori_{P|_{\Lambda_\alpha}}\otimes \ori_{D_\nu^\vee}\bigr).
\]
Combining the preceding displays gives
\[
  \mathfrak o_{\Lambda_\alpha,\nu}
  \simeq
  \ori_{\Lambda_\alpha}\otimes \ori_{D_\nu}
  \otimes \ori_{P|_{\Lambda_\alpha}}\otimes \ori_{D_\nu^\vee}\simeq\ori_{\Lambda_\alpha}\otimes \ori_{P|_{\Lambda_\alpha}}.
\]
Restriction to top-dimensional strata defines a morphism of sheaves
\[
  \rho_\Lambda:\mathcal Z_{\Lambda,\nu}\longrightarrow
  \bigoplus_{\alpha\in A_{\mathrm{top}}}(j_\alpha)_*\mathfrak o_{\Lambda_\alpha,\nu}.
\]
This morphism is injective.
Indeed, a Lagrangian cycle in $\dot T^*Z$ has dimension $\dim Z$; if its restriction to every smooth top-dimensional branch $B_\alpha$ vanishes, then its support is contained in the preimage of the union of lower-dimensional strata of $\Lambda$, which has no local $\dim Z$-dimensional Lagrangian cycle component.
Hence the cycle is zero.
The image of $\rho_\Lambda$ consists of those tuples $(s_\alpha)_\alpha$ satisfying the local cycle condition: in every local top-dimensional chart, the corresponding local Borel--Moore chain has vanishing boundary along the lower-dimensional strata.

A local generator is only available after a local brane trivialization.
Let $U\subset \Lambda_\alpha$ be an open set on which the polarized brane data is trivialized.
Such a trivialization determines a generator
\[
  \epsilon_U:\underline{\ZZ}_U\xrightarrow{\sim}\mathfrak o_{\Lambda_\alpha,\nu}|_U
\]
and a compatible microstalk functor, denoted $\mu_U$.
Changing the local trivialization gives (difference of) Maslov data $\gamma\in\pi_1(\mathrm{LGr}(\RR^{2n}))$. As oriented Lagrangian subspaces are parametrized by $U(n)/SO(n)$, the double cover $U(n)/SO(n)\to U(n)/O(n)$ induces the action of $\gamma$ on the orientation $\mathfrak{or}_{\Lambda}$ as $$\pi_1(\mathrm{LGr}(\RR^{2n}))=\pi_1(U(n)/O(n))\to\ZZ/2.$$
Therefore, the shift of microstalk as in \cite[Prop.~7.5.3]{KS90} is of same parity as the change of orientation, so the product $\chi(\mu_U(-))\epsilon_U$ is an invariant object.

\begin{prop}\label{prop:cc-generic-multiplicity}
For any open subset $V\subset \Lambda$ and any $\cF\in \mush^c_\Lambda(V)$, the image of $\CC_{\Lambda,\nu}([\cF])|_V$ under $\rho_\Lambda$ is a family
\[
  \bigl(c_{\alpha,V}(\cF)\bigr)_{\alpha\in A_{\mathrm{top}}},
  \qquad
  c_{\alpha,V}(\cF)\in \Gamma(V\cap \Lambda_\alpha,\mathfrak o_{\Lambda_\alpha,\nu}).
\]
Moreover, for every open set $U\subset V\cap \Lambda_\alpha$ with a local brane trivialization as above, one has
\[
  c_{\alpha,V}(\cF)|_U
  =
  \chi\bigl(\mu_U(\cF)\bigr)\,\epsilon_U.
\]
\end{prop}

\begin{proof}

Fix $\alpha\in A_{\mathrm{top}}$.
The assertion is local on $V\cap \Lambda_\alpha$, so choose an open subset $U\subset V\cap \Lambda_\alpha$ on which the polarized brane data is trivialized.
Set
\[
  B_U:=p_\alpha^{-1}(U)\subset B_\alpha.
\]
After shrinking $U$ if necessary, choose a conic open neighborhood $\Omega\subset \dot T^*Z$ of $B_U$ and a constructible representative $G\in \Sh^c(Z)$ of $\cF|_U$ in the Nadler--Shende ambient model, with
\[
  \ss(G)\cap \Omega\subset \widehat\Lambda_\nu.
\]
The chosen local brane trivialization gives a local branch generator
\[
  e_{B_U}\in
  \Gamma\bigl(B_U,i_{B_U}^{-1}\underline{\Gamma}_{B_U}(\mathcal L_Z)\bigr)
\]
corresponding to $p_\alpha^{-1}\epsilon_U$ under the identification
\[
  i_{B_U}^{-1}\underline{\Gamma}_{B_U}(\mathcal L_Z)
  \simeq p_\alpha^{-1}\mathfrak o_{\Lambda_\alpha,\nu}|_U.
\]
The generic multiplicity formula for the Kashiwara--Schapira characteristic cycle on the smooth conic Lagrangian branch $B_U$ gives
\[
  \CC(G)|_{B_U}
  =
  \chi\bigl(\mu_{B_U}(G)\bigr)\,e_{B_U};
\]
see \cite[Ch.~IX]{KS90}.
Here $\mu_{B_U}(G)$ is the microstalk local system of $G$ along $B_U$, normalized by the same local brane trivialization.
By the definition of $\mush^c_\Lambda$ through the $\nu$-thickened ambient model, $\mu_{B_U}(G)$ is the pullback of the microstalk local system $\mu_U(\cF)$ on $U$.
Thus
\[
  p_\alpha^{-1}c_{\alpha,V}(\cF)|_U
  =
  p_\alpha^{-1}\bigl(\chi(\mu_U(\cF))\,\epsilon_U\bigr).
\]
Since $p_\alpha:B_U\to U$ has contractible fibers, pullback by $p_\alpha$ is faithful on local systems.
Hence
\[
  c_{\alpha,V}(\cF)|_U
  =
  \chi\bigl(\mu_U(\cF)\bigr)\,\epsilon_U.
\]
\end{proof}

We immediately have the following corollary.

\begin{cor}\label{cor:cc-independent-embedding}
For a fixed stable Lagrangian-normal datum $\nu$, the local systems $\mathfrak o_{\Lambda_\alpha,\nu}$, the subsheaf
\[
  \mathcal Z_{\Lambda,\nu}\subset
  \bigoplus_{\alpha\in A_{\mathrm{top}}}(j_\alpha)_*\mathfrak o_{\Lambda_\alpha,\nu},
\]
and the morphism
\[
  \CC_{\Lambda,\nu}:K_0^{\mathrm{pre}}(\mush^c_\Lambda)\longrightarrow \mathcal Z_{\Lambda,\nu}
\]
are independent of the auxiliary contact embedding used to realize $\nu$, under the canonical identifications supplied by the Nadler--Shende construction.
\end{cor}

\subsection{The characteristic-cycle construction for $\fL$}

Recall from \cref{sec:setup} that $U$ is a Weinstein neighborhood of the RSTZ skeleton
\[
  \fL=(\Lambda\times S^1)\cup_{\partial\Lambda\times S^1\times I}(\partial\Lambda\times D_\disk).
\]
We fix the stable Lagrangian-normal datum $\nu_\fL$ coming from the chosen stably polarized Weinstein structure on $U$. The skeleton $\fL\subset U$ satisfies the assumptions of \cref{subsec:cc-global-ambient}, hence we have the morphism of presheaves on $\fL$
\[
  \CC_{\fL,\nu_\fL}:K_0^{\mathrm{pre}}(\mush^c_\fL)\longrightarrow \mathcal Z_{\fL,\nu_\fL}.
\]
We now fix the coefficient conventions on each top-dimensional capped branch. Recall that $P:=T_{\mathrm{pol}}U$ is the Lagrangian distribution chosen in \cref{subsec:polarization}. The construction of FLTZ skeleton induces a stratification with conic strata. For a top-dimensional stratum $\Lambda_\alpha$, we define a reference coefficient system $\mathfrak r_\alpha$ on $\fL_\alpha$ by
\[
  \mathfrak r_\alpha|_{\Lambda_\alpha\times S^1}
  =
  \mathfrak o_{\Lambda_\alpha}\boxtimes \mathfrak o_{0_{S^1}},
  \qquad
  \mathfrak r_\alpha|_{\Lambda_\alpha^\infty\times D_{\mathrm{disk}}}
  =
  \mathfrak o_{\Lambda_\alpha^\infty}\boxtimes \mathfrak o_{0_{D_\disk}}.
\]
Here $\mathfrak o_{\Lambda_\alpha}=\mathcal L_Q|_{\Lambda_\alpha}$ is the Kashiwara--Schapira coefficient local system of the FLTZ branch, $\mathfrak o_{\Lambda_\alpha^\infty}$ is its radial descent to the Legendrian boundary, and $\mathfrak o_{0_{S^1}}$, $\mathfrak o_{0_{D_\disk}}$ are the coefficient systems of the zero sections in $T^*S^1$ and $T^*D_\disk$. The two product systems are glued on the collar by the zero section collar identification
\[
  0_{D_\disk}|_{S^1\times I}=0_{S^1}\times 0_I.
\]
We now define a comparison
\[
  \theta_{\alpha,P}:
  \mathfrak r_\alpha
  \congto
  \ori_{\sL_\alpha}\otimes\ori_{P|_{\sL_\alpha}}
  \simeq
  \mathfrak{o}_{\sL_\alpha,\nu_\sL}  
\] by declaring isomorphisms on an open cover $\sL_{\alpha,i},i\in{1,2,12}$, where
\[
  \fL_{\alpha,1}\subset\Lambda_\alpha\times S^1\text{: away from conic end or }\varepsilon_2<t<\varepsilon;
\]\[
  \fL_{\alpha,2}\subset\Lambda_\alpha^\infty\times D_\disk\text{: away from conic end or }t<\varepsilon_1;
\]\[
  \fL_{\alpha,12}=
  \Lambda_\alpha^\infty\times S^1\times (\varepsilon_1-\delta,\varepsilon_2+\delta)_t,
\]
as in the construction of $P$, for a small enough $\delta$.
The gluing of such isomorphisms is guarenteed by the construction on $\sL_{\alpha,12}$.
Let
\[
  V_\theta:=T_{\mathrm{vert}}(T^*S^1)=\RR\partial_{p_\theta},
  \qquad
  V_t:=T_{\mathrm{vert}}(T^*(\varepsilon_1,\varepsilon_2))=\RR\partial_{p_t},
  \qquad
  H_t:=\RR\partial_t.
\]
On $\fL_{\alpha,1}$ we have $P|_{\fL_{\alpha,1}}
  =T_{\mathrm{pol}}U_\Lambda\oplus V_\theta$ by construction. Then \[
  \mathfrak o_{\Lambda_\alpha}\boxtimes \mathfrak o_{0_{S^1}}
  \simeq
  \ori_{\Lambda_\alpha\times S^1}
  \otimes
  (\ori_{T_{\mathrm{pol}}U_\Lambda}\boxtimes\ori_{V_\theta})
  \simeq
  \ori_{\sL_{\alpha,1}}\otimes\ori_{P|_{\sL_{\alpha,1}}}.
\]
On $\fL_{\alpha,2}$ we have similarly
\[
  \mathfrak o_{\Lambda_\alpha^\infty}\boxtimes \mathfrak o_{0_{D_\disk}}
  \simeq
  \ori_{\Lambda_\alpha^\infty\times D_\disk}
  \otimes
  \ori_{T_{\mathrm{vert}}F\oplus T_{\mathrm{vert}}(T^*D_\disk)}
  \simeq
  \ori_{\sL_{\alpha,2}}\otimes\ori_{P|_{\sL_{\alpha,2}}}.
\]
On $\sL_{\alpha,12}$ the chosen polarization is
\[
  P|_{\fL_{\alpha,12}}
  =T_{\mathrm{vert}}F\oplus \ell\oplus V_\theta,
\]
where
\[
  \ell|_{t\in[\varepsilon_1,\varepsilon_2]}=\RR v(t),
  \qquad
  v(t)=
  \frac{t-\varepsilon_1}{\varepsilon_2-\varepsilon_1}\partial_t
  +
  \frac{\varepsilon_2-t}{\varepsilon_2-\varepsilon_1}\partial_{p_t}.
\]
We orient $\ell$ by declaring $v(t)$ to be positive. This orientation gives isomorphisms
\[
  R_\ell^H:\ori_{H_t}\xrightarrow{\sim}\ori_\ell,
  \qquad [\partial_t]\longmapsto [v(t)];
\]
\[
  R_\ell^V:\ori_{V_t}\xrightarrow{\sim}\ori_\ell,
  \qquad [\partial_{p_t}]\longmapsto [v(t)].
\]
Near the conic end, radial descent gives
\[
  \mathfrak o_{\Lambda_\alpha}|_{\Lambda_\alpha^\infty\times I}
  \simeq
  \mathfrak o_{\Lambda_\alpha^\infty}\boxtimes \mathfrak o_{I,H},
  \qquad
  \mathfrak o_{I,H}:=\ori_I\otimes\ori_{H_t},
\]
where the factor \(\mathfrak o_{I,H}\) is canonically trivialized by the
positive radial coordinate.  On the disk side,
\[
  \mathfrak o_{0_{D_\disk}}|_{S^1\times I}
  \simeq
  \mathfrak o_{0_{S^1}}\boxtimes\mathfrak o_{I,V},
  \qquad
  \mathfrak o_{I,V}:=\ori_I\otimes\ori_{V_t},
\]
and \(\mathfrak o_{I,V}\) is trivialized by the standard zero-section
orientation of \(0_I\subset T^*I\). The gluing identifies these
two trivialized interval factors. The chosen positive vector \(v(t)\in \ell\) induces
\[
  \id\otimes R_\ell^H:
  \mathfrak o_{I,H}\to \ori_I\otimes\ori_\ell;
\]\[
  \id\otimes R_\ell^V:
  \mathfrak o_{I,V}\to \ori_I\otimes\ori_\ell.
\]
Both send the chosen positive interval generator to
\([\partial_t]\otimes[v(t)]\). Hence the comparisons on $\sL_{\alpha,1}$ and $\sL_{\alpha,2}$ are compatible with $\sL_{\alpha,12}$ respectively.

\subsection{A cap formula for zero section objects}
\label{sec:cap}
For the mirror Lagrangian skeleton
\[
  i_\Lambda:\Lambda\hookrightarrow T^*Q,
\]
recall that the global section of the characteristic cycle map
\[
  \CC_\Lambda:K_0^{\mathrm{pre}}(\mush^c_\Lambda)\longrightarrow \mathcal Z_\Lambda:=i_\Lambda^{-1}\underline{\Gamma}_{\Lambda}(\mathcal L_Q)
\]
is the classical characteristic cycle under the equivalence $\mush^c_\Lambda(\Lambda)\simeq \Sh^c_\Lambda(Q)$.

Restriction to top-dimensional strata gives an injective morphism
\[
  \rho_\Lambda:\mathcal Z_\Lambda\hookrightarrow
  \bigoplus_{\alpha\in A_{\mathrm{top}}}(j_\alpha)_*\mathfrak o_{\Lambda_\alpha},
  \qquad
  \mathfrak o_{\Lambda_\alpha}:=\mathcal Z_\Lambda|_{\Lambda_\alpha}.
\]
Here $A_{\mathrm{top}}$ indexes the connected top-dimensional branches. For every $\alpha\in A_{\mathrm{top}}$, we have a morphism $$s_\alpha\mapsto s_\alpha^\infty:\qquad \Gamma(\Lambda_\alpha,\mathfrak o_{\Lambda_\alpha})\to \Gamma(\Lambda_\alpha^\infty,\mathfrak o_{\Lambda_\alpha^\infty})$$ induced by
\[
  \mathfrak o_{\Lambda_\alpha}|_{\RR_{>0}\cdot\Lambda_\alpha^\infty}
  \simeq
  \pi_{\alpha,\infty}^{-1}\mathfrak o_{\Lambda_\alpha^\infty},
  \qquad
  \pi_{\alpha,\infty}:\RR_{>0}\cdot\Lambda_\alpha^\infty\to\Lambda_\alpha^\infty.
\]

We first define the reference cap on a single top-dimensional branch. For
$
  s_\alpha\in\Gamma(\Lambda_\alpha,\mathfrak o_{\Lambda_\alpha}),
$
let
\[
  \operatorname{cap}_\alpha(s_\alpha)
  \in
  \Gamma(\fL_\alpha,\mathfrak r_\alpha)
\]
be the section characterized by
\[
  \operatorname{cap}_\alpha(s_\alpha)|_{\Lambda_\alpha\times S^1}
  =
  s_\alpha\boxtimes [0_{S^1}],
\]
and
\[
  \operatorname{cap}_\alpha(s_\alpha)|_{\Lambda_\alpha^\infty\times D_\disk}
  =
  s_\alpha^\infty\boxtimes [0_{D_\disk}].
\]
These two restrictions agree on the collar by the definition of $\mathfrak r_\alpha$. Define the actual capped coefficient by
\[
  -\theta_{\alpha,P}
  \bigl(\operatorname{cap}_\alpha(s_\alpha)\bigr)
  \in
  \Gamma(\fL_\alpha,\mathfrak o_{\fL_\alpha,\nu_\fL}).
\]
Now let $c\in\mathcal Z_\Lambda(\Lambda)$ and write
\[
  \rho_\Lambda(c)=(c_\alpha)_{\alpha\in A_{\mathrm{top}}}.
\]
The family $-\theta_{\alpha,P}\bigl(\operatorname{cap}_\alpha(s_\alpha)\bigr)$ satisfies the local cycle condition on $\fL$. Indeed, it suffices to ckeck on $\Lambda\times S^1$ and $\Lambda^\infty\times D_{\mathrm{disk}}$ since the condition is local. In the sheaf $\mathfrak{r}_\alpha$ it is obtained from the top-dimensional chain representing $c$ by exterior product with the zero section cycle on $S^1$ over the cylindrical piece and with the zero section cycle on $D_\disk$ over the cap piece. The boundary operator satisfies Leibniz rule under exterior product, and the zero section cycles on $S^1$ and $D_\disk$ are closed. On the collar the two product descriptions agree by construction of $\mathfrak r_\alpha$. Hence the local boundary of the capped family is the capped image of the local boundary of $c$, which vanishes by assumptions. Since
\[
  \rho_\fL:\mathcal Z_{\fL,\nu_\fL}\hookrightarrow
  \bigoplus_{\alpha\in A_{\mathrm{top}}}[\fL_\alpha\hookrightarrow \fL]_*\mathfrak o_{\fL_\alpha,\nu_\fL}
\]
is injective with image cut out by the local cycle condition, there is a unique section
\[
  \operatorname{Cap}(c)\in\mathcal Z_{\fL,\nu_\fL}(\fL)
\]
such that
\[
  \rho_\fL\bigl(\operatorname{Cap}(c)\bigr)
  =
  -\theta_{\alpha,P}
  \bigl(\operatorname{cap}_\alpha(s_\alpha)\bigr).
\]
This defines a homomorphism
\[
  \operatorname{Cap}:\mathcal Z_\Lambda(\Lambda)\longrightarrow\mathcal Z_{\fL,\nu_\fL}(\fL).
\]

\begin{theorem}\label{thm:cap-formula}
Let $\iota:S\hookrightarrow \cX^\circ$ be the zero section inclusion and $\kappa_S,\kappa_{\cX^\circ}$ be the equivalences in \cref{sec:cat_eq}. For any $G\in\Coh(S)$%
, we have:
\[
  \CC_{\fL,\nu_\fL}\bigl(\kappa_{\cX^\circ}(\iota_*G)\bigr)=
  \operatorname{Cap}\Bigl(\CC_\Lambda\bigl(\kappa_S(G)\bigr)\Bigr)\in\mathcal{Z}_{\fL,\nu_\fL}(\fL).
\]
\end{theorem}
\begin{proof}
Let
\[
  s:S\simeq S\times\{1\}\hookrightarrow X=S\times\GG_m
\]
be the closed embedding into the cylinder, and let $i_P:P\hookrightarrow S$ be the toric boundary divisor. Post-composing the B-side pushout square \eqref{eqn:B-side-colim} with $\iota^*$ and taking right adjoints gives the square
\[
\begin{tikzcd}[column sep=huge]
\Coh(D) & \Coh(H) \arrow[l, "i_*"] \\
\Coh(X) \arrow[u, "k^*"] & \Coh(S). \arrow[u, "{i_P^*[-1]}"] \arrow[l, "{s_*[-1]}"]
\end{tikzcd}
\]
Here the displayed adjoints are the closed-embedding duality functors already used in the construction of \eqref{eqn:B-side-colim}.

Set $\mathcal F:=\kappa_{\cX^\circ}(\iota_*G)\in\mush^w_\fL(\fL)$ and let $\kappa_1$ and $\kappa_2$ denote the equivalences on the two pieces $\fL_1=\Lambda\times S^1$ and $\fL_2=\partial\Lambda\times D_\disk$. The preceding square gives
\[
  \mathcal F|_{\fL_1}
  \simeq
  \kappa_1(s_*G[-1]),
  \qquad
  \mathcal F|_{\fL_2}
  \simeq
  \kappa_2(i_P^*G[-1]).
\]
Under the coherent-constructible correspondence, $s_*$ is mirrored by exterior product with the constant sheaf on $S^1$, and $i_P^*$ is mirrored by microlocal restriction from the FLTZ skeleton to its Legendrian boundary; see \cite[\S7.2]{shende-gammage}.

Put
\[
  c:=\CC_\Lambda(\kappa_S(G)),
  \qquad
  \rho_\Lambda(c)=(c_\alpha)_\alpha.
\]
The characteristic cycle is additive in $K_0$, hence the shift $[-1]$ multiplies every top-dimensional coefficient by $-1$. This proves $\CC_{\fL,\nu_\fL}\bigl(\kappa_{\cX^\circ}(\iota_*G)\bigr)$ restricts to the components
\[
  -c_\alpha\boxtimes [0_{S^1}],\qquad
  -c_\alpha^\infty\boxtimes [0_{D_\disk}].
\]
As discussed above, the components glue to a section
\[
  -\operatorname{cap}_\alpha(c_\alpha)
  \in
  \Gamma(\fL_\alpha,\mathfrak r_\alpha).
\]
Under the equivalence $\theta_{\alpha,P}:
  \mathfrak r_\alpha
  \xrightarrow{\sim}
  \mathfrak{o}_{\sL_\alpha,\nu_\sL},$ this matches our definition of $\operatorname{Cap}$.
\end{proof}

\subsection{Examples}
\label{subsec:cc-examples}

We spell out a few basic examples of the cap formula.  Let
\[
  d:=\dim_\C S,\qquad N:=\dim_\C K_S=d+1.
\]
For an object $G\in\Coh(S)$%
, we call
$\CC_{\fL,\nu_\fL}\bigl(\kappa_X(\iota_*G)\bigr)$
the mirror cycle of $G$ in $\fL$.

\begin{exam}
  For $G=\cO_S$, the characteristic cycle
  $\CC_\Lambda(\kappa_S(\cO_S))$ is the standard vector space $\bR^d$. Capping gives
  \[
    (\mathbb R^d\times S^1)\cup_{S^{d-1}\times S^1}(S^{d-1}\times D_\disk)
    \cong S^{d+1}=S^N.
  \]
  Thus the mirror cycle of $\cO_S$ is homeomorphic to $S^N$.
\end{exam}

\begin{exam}
  Let $p\in S$ be a point and take $G=\cO_p$.  Under the
  coherent-constructible correspondence, $\kappa_S(\cO_p)$ has characteristic
  cycle supported on the zero section torus $Q\simeq (S^1)^d\subset\Lambda$.
  This branch has no boundary at infinity, so the cap operation is simply
  product with the extra circle:
  \[
    Q\times S^1\simeq (S^1)^{d+1}=(S^1)^N.
  \]
  Hence the mirror cycle of a point sheaf is homeomorphic to $(S^1)^N$.
\end{exam}

\begin{exam}
  Suppose $\dim_\C S=2$, so that $K_S$ is a toric Calabi--Yau threefold.
  For any line bundle $G$ on $S$, the capped FLTZ cycle above is the same cycle
  attached to the corresponding $K$-theoretic framing $\mathrm{mir}^c(G)$ in \cite{FLYZ-remodeling-descendants}, i.e. $\CC_{\fL,\nu_\fL}\bigl(\kappa_X(\iota_*G)\bigr)=\mathrm{mir}^c(G)$. We omit the details of the verification here. Therefore, with the correct choice of an actual orientation, we have (\cite[Thm.~6.8]{FLYZ-remodeling-descendants}, Hosono's conjecture \cite{Hos06})
  \[
  \int_{\CC_{\fL,\nu_\fL}\bigl(\kappa_X(\iota_*G)\bigr)} \Omega = Z^c(G),
  \]
  where $\Omega$ is the Calabi--Yau form on $\cY$, and $Z^c(G)$ is the \emph{Gromov--Witten A-model central charge} defined using Gamma classes. Thus our mirror cycle matches the correspondence from Hosono's conjecture for line bundles supported on $S$. We expect $\CC_{\fL,\nu_\fL}\bigl(\kappa_X(\iota_*G)\bigr)=\mathrm{mir}^c(G)$ for all coherent sheaves $G$ supported on $S$ but did not carry out the computation for the remaining generating cases (line bundles supported on a toric line).
\end{exam}

\bibliographystyle{alpha}
\bibliography{HMS}

\end{document}